\documentclass{elsart3-1}

\usepackage{amssymb,mathtools}
\usepackage{url}

\usepackage[english,french]{babel}

\newtheorem{theorem}{Theorem}[section]
\newtheorem{lemma}[theorem]{Lemma}
\newtheorem{e-proposition}[theorem]{Proposition}

\newtheorem{e-definition}[theorem]{Definition\rm}
\newtheorem{remark}{\it Remark\/}
\newtheorem{example}{\it Example\/}

\renewcommand{\theequation}{\arabic{equation}}
\def\og{\leavevmode\raise.3ex\hbox{$\scriptscriptstyle\langle\!\langle$~}}
\def\fg{\leavevmode\raise.3ex\hbox{~$\!\scriptscriptstyle\,\rangle\!\rangle$}}

\usepackage[utf8]{inputenc}
\usepackage[T1]{fontenc}
\usepackage{csquotes}
\usepackage{array,booktabs,multirow}
\usepackage{multicol} 
\usepackage{longtable} 
\usepackage[small,labelsep=newline,singlelinecheck=off]{caption}
\usepackage{amsmath}
\usepackage{xcolor}
\newcommand{\faded}{\color{gray}}
\newcommand{\pileupmath}[2][c]{\ensuremath{\begin{array}{@{}#1@{}}\strut#2\strut\end{array}}}
\newcommand{\pileuptext}[1]{\text{\begin{tabular}{@{}c@{}}#1\end{tabular}}}
\newcommand{\notapplicable}{\text{\tiny N.A.}}
\usepackage{enumitem} 
\newlist{hypothenum}{enumerate}{3}
\setlist[hypothenum,1]{label=(\roman*)}
\newcommand{\eps}{\varepsilon}
\newcommand{\setsuch}[2]{\left\{ #1 \; \middle| \; #2 \right\}}

\newcommand{\transpose}[1]{#1^{\mathrm t}}

\newcommand{\ZZ}{\mathbb{Z}}
\newcommand{\CC}{\mathbb{C}}

\DeclareMathOperator{\rank}{rank}

\DeclareMathOperator{\Id}{Id}

\DeclareMathOperator{\Span}{Span}

\DeclareMathOperator{\SL}{SL}

\newcommand{\ie}{i.e.\ }
\newcommand{\eg}{e.g.\ }
\newcommand{\lie}{\mathfrak}
\newcommand{\abs}[1]{\lvert #1\rvert}

  \usepackage{zref-abspage}[2010/03/26]
  \makeatletter
  \newcounter{topic@label}
  \renewcommand*{\thetopic@label}{topic@\the\value{topic@label}}
  \global\let\topic@previous\relax
  \global\let\lasttopic\relax
  \newcommand*{\topic}[1]{%
    \begingroup
      \def\topic@put{\topicformat{#1}}%
      \edef\topic@previouslabel{\thetopic@label}%
      \stepcounter{topic@label}%
      \zref@labelbyprops{\thetopic@label}{abspage}%
      \def\topic@current{#1}%
      \ifx\topic@current\topic@previous
        \zifrefundefined{\topic@previouslabel}{%
          \topic@put
        }{%
          \zifrefundefined{\thetopic@label}{%
            \topic@put
          }{%
            \ifnum\zref@extractdefault{\topic@previouslabel}{abspage}{0}=%
                  \zref@extractdefault{\thetopic@label}{abspage}\relax
            \else
              {\topic@put {\tiny { (cont.)}}}
            \fi
          }%
        }%
      \else
        \topic@put
      \fi
      \global\let\topic@previous\topic@current
    \endgroup
    \gdef\lasttopic{\topic{#1}}%
  }
  \newcommand*{\topicformat}[1]{#1}
  \makeatother

\providecommand{\qedsymbol}{\square}
\providecommand{\qedhere}{\mbox{}\hfill$\qedsymbol$\begingroup\toks0\expandafter{\qedsymbol}\xdef\qedsymbol{\gdef\noexpand\qedsymbol{\the\toks0}}\endgroup}
\newenvironment{proof}[1][\unskip]{\trivlist\item\textbf{Proof #1. }\ignorespaces}{\mbox{}\hfill$\qedsymbol$\endtrivlist}
\journal{the Acad\'emie des sciences}
\begin{document}
\date{June 1, 2018}
\centerline{Lie Algebras}
\begin{frontmatter}

\selectlanguage{english}
\title{Action of Weyl group on zero-weight space}

\selectlanguage{english}

\author[bruno]{Bruno Le Floch},
\ead{blefloch@princeton.edu}
\author[ilia]{Ilia Smilga}
\ead{ilia.smilga@normalesup.org}
\ead[url]{http://gauss.math.yale.edu/\string~is362/index.html}

\address[bruno]{Princeton Center for Theoretical Science, Princeton NJ 08544, USA}
\address[ilia]{Yale University Mathematics Department, PO Box 208283, New Haven CT 06520-8283, USA}

\medskip
\begin{center}\small
  Received ?????; accepted after revision ?????\\
  Presented by ?????
\end{center}

\begin{abstract}
\selectlanguage{english}%
For any simple complex Lie group we classify irreducible finite-dimensional representations~$\rho$ for which the longest element~$w_0$ of the Weyl group acts nontrivially on the zero-weight space.  Among irreducible representations that have zero among their weights, $w_0$ acts by $\pm\Id$ if and only if the highest weight of~$\rho$ is a multiple of a fundamental weight, with a coefficient less than a bound that depends on the group and on the fundamental weight.
{\it To cite this article: B. Le~Floch, I. Smilga, C. R. Acad.\@ Sci.\@ Paris, Ser. I ??? (2018).}

\vskip 0.5\baselineskip

\selectlanguage{french}
\noindent{\bf R\'esum\'e} \vskip 0.5\baselineskip \noindent
{\bf Action du groupe de Weyl sur l'espace de poids nul. }
Pour tout groupe de Lie complexe simple nous classifions les repr\'esentations irr\'eductibles~$\rho$ de dimension finie telles que le plus long mot~$w_0$ du groupe de Weyl agisse non-trivialement sur l'espace de poids nul.  Parmi les repr\'esentations irr\'eductibles dont z\'ero est un poids, $w_0$ agit par $\pm\Id$ si et seulement si le plus haut poids de~$\rho$ est un multiple d'un poids fondamental, avec un coefficient plus petit qu'une borne qui d\'epend du groupe et du poids fondamental.
{\it Pour citer cet article~: B. Le~Floch, I. Smilga, C. R. Acad.\@ Sci.\@ Paris, Ser. I ??? (2018).}

\par
\end{abstract}
\par
\end{frontmatter}


\selectlanguage{english}
\section{\label{sec:intro}Introduction and main theorem}

Consider a reductive complex Lie algebra~$\lie{g}$. Let $\tilde{G}$ be the corresponding simply-connected Lie group.

We choose in~$\lie{g}$ a Cartan subalgebra $\lie{h}$.
Let $\Delta$~be the set of roots of~$\lie{g}$ in~$\lie{h}^*$. We call $\Lambda$ the root lattice, \ie the abelian subgroup of~$\lie{h}^*$ generated by~$\Delta$. We choose in~$\Delta$ a system $\Delta^+$ of positive roots; let $\Pi = \{\alpha_1, \ldots, \alpha_r\}$ be the set of simple roots in~$\Delta^+$. Let~$\varpi_1, \ldots, \varpi_r$ be the corresponding fundamental weights.
Let $W \coloneqq N_{\tilde{G}}(\lie{h})/Z_{\tilde{G}}(\lie{h})$ be the Weyl group, and let $w_0$ be its longest element (defined by $w_0(\Delta^+) = -\Delta^+$).

For each simple Lie algebra, we call $(e_1, e_2, \ldots)$ the vectors called $(\eps_1, \eps_2, \ldots)$ in the appendix to \cite{BouGAL456}, which form a convenient basis of a vector space containing $\lie{h}^*$. Throughout the paper, we use the Bourbaki conventions \cite{BouGAL456} for the numbering of simple roots and their expressions in the coordinates $e_i$.

In the sequel, all representations are supposed to be complex and finite-dimensional. We call $\rho_\lambda$ (resp.~$V_\lambda$) the irreducible representation of~$\lie{g}$ with highest weight~$\lambda$ (resp.\ the space on which it acts). Given a representation $(\rho, V)$ of~$\lie{g}$, we call $V^\lambda$ the weight subspace of~$V$ corresponding to the weight~$\lambda$.

\begin{e-definition}\label{def:radical}
  We say that a weight~$\lambda \in \lie{h}^*$ is \emph{radical} if $\lambda \in \Lambda$.
\end{e-definition}
\begin{remark}\label{rem:radical}
An irreducible representation $(\rho,V)$ has non-trivial zero-weight space~$V^0$ if and only if its highest weight is radical.
\end{remark}

\begin{e-definition}
Let $(\rho, V)$ be a representation of~$\lie{g}$. The action of $W = N_{\tilde{G}}(\lie{h})/Z_{\tilde{G}}(\lie{h})$ on~$V^0$ is well-defined, since $V^0$ is by definition fixed by $\lie{h}$, hence by~$Z_{\tilde{G}}(\lie{h})$. Thus $w_0$ induces a linear involution on~$V^0$. Let $p$ (resp.~$q$) be the dimension of the subspace of~$V^0$ fixed by~$w_0$ (resp.\ by~$-w_0$). We say that $(p, q)$ is the \emph{$w_0$-signature} of the representation~$\rho$ and that the representation is:
\begin{itemize}
\item \emph{$w_0$-pure} if $pq=0$ (\emph{of sign} $+1$ if $q=0$ and \emph{of sign} $-1$ if $p=0$);
\item \emph{$w_0$-mixed} if $pq>0$.
\end{itemize}
\end{e-definition}
\begin{remark}\label{rem:isogeny_does_not_matter}
Replacing $\tilde{G}$ by any other connected group~$G$ with Lie algebra~$\lie{g}$ (with a well-defined action on~$V$) does not change the definition. Indeed the center of~$\tilde{G}$ is contained in~$Z_{\tilde{G}}(\lie{h})$ so acts trivially on~$V^0$. 
\end{remark}
Our interest in this property originates in the study of free affine groups acting properly discontinuously (see \cite{Smi16b}). We prove the following complete classification. To the best of our knowledge, this specific question has not been studied before; see \cite{Hum14} for a survey of prior work on related, but distinct, questions about the action of the Weyl group on the zero-weight space.

\begin{theorem}
\label{main_theorem}
Let $\lie{g}$ be any simple complex Lie algebra; let $r$~be its rank.
For every index $1\leq i\leq r$, we denote by~$p_i$ the smallest positive integer such that $p_i \varpi_i \in \Lambda$.
For every such $i$, let the ``maximal value'' $m_i \in \ZZ_{\geq 0} \cup \{\infty\}$ and the ``sign'' $\sigma_i \in \{\pm1\}$ be as given in Table~\ref{mui-ni-table} on page~\pageref{mui-ni-table}.

Let $\lambda$ be a dominant weight.
\begin{hypothenum}
\item\label{main nonradical} If $\lambda \not\in \Lambda$, then the $w_0$-signature of the representation~$\rho_\lambda$ is $(0,0)$.
\item\label{main pure radical} If $\lambda = k p_i \varpi_i$ for some $1\leq i\leq r$ and $0\leq k \leq m_i$, then $\rho_\lambda$ is $w_0$-pure of sign $(\sigma_i)^k$.
\item\label{main mixed} Finally, if $\lambda \in \Lambda$ but is not of the form $\lambda = k p_i \varpi_i$ for any $1\leq i\leq r$ and $0\leq k\leq m_i$, then $\rho_\lambda$ is $w_0$-mixed.
\end{hypothenum}
\end{theorem}

\begin{example}\label{example:sl2}
  Any irreducible representation of $\SL(2,\CC)$ is isomorphic to $S^k \CC^2$ (the $k$-th symmetric power of the standard representation) for some $k\in\ZZ_{\geq 0}$.  Its $w_0$-signature is $(0,0)$ if $k$ is odd, $(1,0)$ if $k$ is divisible by $4$ and $(0,1)$ if $k$ is $2$ modulo $4$.  This confirms the $A_1$ entries $(p_1,m_1,\sigma_1)=(2,\infty,-1)$ of Table~\ref{mui-ni-table}.
\end{example}

Table~\ref{mui-ni-table} also gives the values of~$p_i$. These are not a new result: they can be found immediately from the congruence relations satisfied by the root lattice (given in Table~\ref{root_lattice_congruences} in the appendix), which are ultimately easy to compute from the known descriptions of the simple roots and fundamental weights (given \eg in \cite{BouGAL456}).

Point~\ref{main nonradical} is an immediate consequence of Remark~\ref{rem:radical}.

For point~\ref{main pure radical}, we show in Section~\ref{sec:proof pure} that certain symmetric and antisymmetric powers of defining representations of classical groups are $w_0$-pure, and that almost all representations listed in point~\ref{main pure radical} are sub-representations of these powers.  The finitely many exceptions are treated by an algorithm described in Section~\ref{sec:algo}.

\begin{table}
  \caption{\label{mui-ni-table}Values of $(p_i,m_i,\sigma_i)$ for simple Lie algebras. Theorem~\ref{main_theorem} states that among irreducible representations with a highest weight~$\lambda$ that is radical, only those with $\lambda$ of the form $k p_i \varpi_i$ with $k \leq m_i$ are $w_0$-pure, with a sign given by $\sigma_i^k$.
    We write $\notapplicable$ for $\sigma_i$~sign entries that are not defined due to $m_i=0$.
    Since $A_1\simeq B_1\simeq C_1$ and $B_2\simeq C_2$ and $A_3\simeq D_3$, the results match up to reordering simple roots (namely reordering $i=1,\ldots,r$).
  }
  \centering\bigskip
  \begin{tabular}[t]{llr>{\hspace{1em}}c>{\hspace{1em}}c>{\hspace{1em}}c}
    \toprule
    & \multicolumn{2}{l}{Values of $i$ and $r$} & $p_i$ & $m_i$ & $\sigma_i$ \\
    \midrule
    \multirow{3}{3em}{$A_{r\geq 1}$}
    & $i=1$ or $r$             &       & $r+1$ & $\infty$ & $(-1)^{\lfloor(r+1)/2\rfloor}$ \\
    \cmidrule{2-6}
    & $1<i<r$ & \pileupmath{r=3\\r>3} & $\frac{r+1}{\gcd(i,r+1)}$ & \pileupmath{\infty\\0} & \pileupmath{+1\\\notapplicable} \\
    \midrule
    \multirow{6}{3em}[-1.5ex]{$B_{r\geq 1}$}
    & $i=1$                  & $r>1$ & $1$ & $\infty$ & \multirow{6}{*}[-1.5ex]{$(-1)^{ri-\lfloor i/2\rfloor}$} \\
    \cmidrule{2-5}
    & $i=2$ & $r>2$ & $1$ & $2$ & \\
    \cmidrule{2-5}
    & $2<i<r$ & & $1$ & $1$ & \\
    \cmidrule{2-5}
    & \multirow{2}{*}{$i=r$} & $r=1,2$ & \multirow{2}{*}{$2$} & $\infty$ & \\
    & & $r>2$ & & $1$ & \\
    \midrule
    \multirow{6}{3em}[-1.5ex]{$C_{r\geq 1}$}
    & $i=1$ & & $2$ & $\infty$ & $-1$ \\
    \cmidrule{2-6}
    & $i=2$ & \pileupmath{r=2\\r>2} & $1$ & \pileupmath{\infty\\2} & $+1$ \\
    \cmidrule{2-6}
    & $i$ odd $>2$ & \pileupmath[r]{i=r=3\\r>3} & $2$ & \pileupmath{1\\0} & \pileupmath{-1\\\notapplicable} \\
    \cmidrule{2-6}
    & $i$ even $>2$ & \pileupmath[r]{i=r=4\\r>4} & $1$ & \pileupmath{2\\1} & $+1$ \\
    \midrule
    \multirow{5}{3em}[-1.5ex]{$D_{r\geq 3}$\\$r$ odd}
    & $i=1$ & & $2$ & $\infty$ & $+1$ \\
    \cmidrule{2-6}
    & $1<i<r-1$ & \pileuptext{$i$ even\\$i$ odd} & \pileupmath{1\\2} & $0$ & \notapplicable \\
    \cmidrule{2-6}
    & $i=r-1$ or $r$ & \pileupmath{r=3\\r>3} & $4$ & \pileupmath{\infty\\0} & \pileupmath{+1\\\notapplicable} \\
    \midrule
    \multirow{5}{3em}[-1.5ex]{$D_{r\geq 4}$\\$r$ even}
    & $i=1$ & & $2$ & $\infty$ & $+1$ \\
    \cmidrule{2-6}
    & $i=2$ & & $1$ & $2$ & $-1$ \\
    \cmidrule{2-6}
    & $2<i<r-1$ & \pileuptext{$i$ odd\\$i$ even} & \pileupmath{2\\1} & \pileupmath{0\\1} & \pileupmath{\notapplicable\\(-1)^{i/2}} \\
    \cmidrule{2-6}
    & $i=r-1$ or $r$ & \pileupmath{r=4\\r>4} & $2$ & \pileupmath{\infty\\1} & $(-1)^{r/2}$ \\
    \bottomrule
  \end{tabular}
  \hspace{.04\linewidth}%
  \begin{tabular}[t]{l>{\hspace{1em}}>$l<$>{\hspace{1em}}>$c<$>{\hspace{1em}}>$c<$>{\hspace{1em}}>$c<$}
    \toprule
    & \text{Values of }i & p_i & m_i & \sigma_i \\
    \midrule
    \multirow{2}{*}{$E_6$} & i=1,3,5,6 & 3 & 0 & \notapplicable \\
    & i=2,4 & 1 & 0 & \notapplicable \\
    \midrule
    \multirow{5}{*}{$E_7$}
    & i=1 & 1 & 2 & -1 \\
    & i=2,5 & 2 & 0 & \notapplicable \\
    & i=3,4 & 1 & 0 & \notapplicable \\
    & i=6 & 1 & 1 & +1 \\
    & i=7 & 2 & 1 & -1 \\
    \midrule
    \multirow{3}{*}{$E_8$}
    & i=1 & 1 & 1 & +1 \\
    & 1<i<8 & 1 & 0 & \notapplicable \\
    & i=8 & 1 & 2 & -1 \\
    \midrule
    \multirow{3}{*}{$F_4$}
    & i=1 & 1 & 2 & -1 \\
    & i=2,3 & 1 & 0 & \notapplicable \\
    & i=4 & 1 & 2 & +1 \\
    \midrule
    $G_2$
    & i=1,2 & 1 & 2 & -1 \\
    \bottomrule
  \end{tabular}
\end{table}

For point~\ref{main mixed} we prove in Section~\ref{sec:ideal} that the set of highest weights of $w_0$-mixed representations of a given group is an ideal of the monoid of dominant radical weights. For any fixed group, this reduces the problem to checking $w_0$-mixedness of finitely many representations. In Section~\ref{sec:inductionrule}, we immediately conclude for exceptional groups and for low-rank classical groups by the algorithm of Section~\ref{sec:algo}; we proceed by induction on rank for the remaining classical groups.

\section{\label{sec:algo}An algorithm to compute explicitly the $w_0$-signature of a given representation}

\begin{e-proposition}\label{sl2_copies}
  Any simple complex Lie group~$G$ admits a reductive subgroup~$S$ whose Lie algebra is isomorphic to $\lie{sl}(2,\CC)^s \times \CC^t$, where $(t, s)$ is the $w_0$-signature of the adjoint representation of~$G$, and whose $w_0$ element is compatible with that of~$G$, in the sense that some representative of the $w_0$ element of~$S$ is a representative of the $w_0$ element of~$G$. This subgroup~$S$ can be explicitly described.
\end{e-proposition}
Note that $s+t = r$ (the rank of~$G$) and that $t = 0$ except for $A_n$ ($t = \lfloor \frac{n}{2} \rfloor$), $D_{2n+1}$ ($t = 1$) and $E_6$ ($t = 2$).

\begin{proof}
  Let $(\lie{h}^*)^{-w_0}$ be the $-1$ eigenspace of~$w_0$.
  Recall that two roots $\alpha$ and~$\beta$ are called \emph{strongly orthogonal} if $\langle \alpha, \beta \rangle = 0$ and neither $\alpha + \beta$ nor $\alpha - \beta$ is a root.
  Table~\ref{tab:orthoroots} exhibits pairwise strongly orthogonal roots $\{\alpha_1, \ldots, \alpha_s\} \subset \Delta$ spanning $(\lie{h}^*)^{-w_0}$ as a vector space.
  (Our sets are conjugate to those of \cite{AK84} but these authors did not need the elements~$w_0$ to match.)
  We then set
  \[\lie{s} \coloneqq \lie{h} \oplus \bigoplus_{i=1}^{s} \left( \lie{g}^{\alpha_i} \oplus \lie{g}^{-\alpha_i} \right),\]
  where $\lie{g}^\alpha$ denotes the root space corresponding to~$\alpha$. This is a Lie subalgebra of~$\lie{g}$, as follows from $[\lie{g}^\alpha, \lie{g}^\beta] \subset \lie{g}^{\alpha+\beta}$ and from strong orthogonality of the $\alpha_i$. It is isomorphic to $\lie{sl}(2,\CC)^s \times \CC^t$, because it has Cartan subalgebra~$\lie{h}$ of dimension $r = s+t$ and a root system of type $A_1^s$. We define $S$ to be the connected subgroup of~$G$ with algebra~$\lie{s}$.

  Let $\overline{\sigma_i}\coloneqq\exp[\tfrac{\pi}{2}(X_{\alpha_i}-Y_{\alpha_i})]\in S$, where for every $\alpha$, $X_\alpha$ and $Y_\alpha$ denote the elements of~$\lie{g}$ introduced in \cite[Theorem 7.19]{Hall15}.
  We claim that $\overline{\sigma}\coloneqq\prod_i\overline{\sigma_i}$ is a representative of the $w_0$ element of~$S$ and of the $w_0$ element of~$G$.
  By \cite[Proposition 11.35]{Hall15}, $\overline{\sigma_i}$ is a representative of the reflection $s_{\alpha_i}$, which shows the first statement.
  Now since the $\alpha_i$ are orthogonal, the product of $s_{\alpha_i}$ acts by $-\Id$ on their span $(\lie{h}^*)^{-w_0}$ and acts trivially on its orthogonal complement, like $w_0$.
  \begin{table}[b]
    \caption{\label{tab:orthoroots}Sets of strongly orthogonal roots that span the vector space $(\lie{h}^*)^{-w_0}$.  We chose them among the positive roots.}
    \centering\bigskip
    \begin{tabular}{c}
      \toprule
      \begin{tabular}[t]{l<{:}l}
        $A_n$
        & $\setsuch{e_i - e_{n+2-i}}{1 \leq i \leq \lfloor (n+1)/2 \rfloor}$ \\
        $B_{2n}$
        & $\setsuch{e_{2i-1} \pm e_{2i}}{1 \leq i \leq n}$ \\
        $B_{2n+1}$
        & $\setsuch{e_{2i-1} \pm e_{2i}}{1 \leq i \leq n} \cup \{e_{2n+1}\}$ \\
        $C_n$
        & $\setsuch{2e_i}{1 \leq i \leq n}$ \\
        $D_n$
        & $\setsuch{e_{2i-1} \pm e_{2i}}{1 \leq i \leq\lfloor n/2\rfloor}$ \\
      \end{tabular}\hspace{3em}
      \begin{tabular}[t]{l<{:}l}
        $E_6$
        & $\begin{array}[t]{@{}l@{}}\{- e_1 + e_4,\; - e_2 + e_3,\\\quad\pm\frac{1}{2}(e_1 + e_2 + e_3 + e_4) + \frac{1}{2}(e_5 - e_6 - e_7 + e_8)\}\end{array}$
        \\
        $E_7$
        & $\{\pm e_1 + e_2,\; \pm e_3 + e_4,\; \pm e_5 + e_6,\; - e_7 + e_8\}$ \\
        $E_8$
        & $\{\pm e_1 + e_2,\; \pm e_3 + e_4,\; \pm e_5 + e_6,\; \pm e_7 + e_8\}$ \\
        $F_4$
        & $\{e_1 \pm e_2,\; e_3 \pm e_4\}$ \\
        $G_2$
        & $\{e_1 - e_2,\; - e_1 - e_2 + 2e_3\}$ \\
      \end{tabular}
      \\\bottomrule
    \end{tabular}
  \end{table}
\end{proof} 

Then the $w_0$-signature of any representation~$\rho$ of $G$ is equal to that of its restriction~$\rho|_S$ to $S$.  We use branching rules to decompose $\rho|_S=\oplus_i\rho_i$ into irreducible representations of~$S$.  The total $w_0$-signature is then the sum of those of the~$\rho_i$.  Each $\rho_i$ is a tensor product $\rho_{i,1}\otimes\dots\otimes\rho_{i,s}\otimes\rho_{i, \operatorname{Ab}}$, where $\rho_{i,j}$ for $1 \leq j \leq s$ is an irreducible representation of the factor $\lie{s}_j\simeq\lie{sl}(2,\CC)$, and $\rho_{i, \operatorname{Ab}}$ is an irreducible representation of the abelian factor isomorphic to~$\CC^t$.  The $w_0$-signature of~$\rho_i$ is then the ``product'' of those of these factors, according to the rule $(p, q)\otimes(p', q') = (pp' + qq', pq' + qp')$. The $w_0$-signatures of all irreducible representations of~$\lie{sl}(2,\CC)$ have been described in Example~\ref{example:sl2}; the $w_0$-signature of $\rho_{i, \operatorname{Ab}}$ is just $(1, 0)$ if the representation is trivial and $(0, 0)$ otherwise.

Branching rules are provided by several software packages.  We implemented our algorithm separately in LiE \cite{LiE} and in Sage \cite{SageMath}. In Sage, we used the Branching Rules module \cite{SagBR}, largely written by Daniel Bump. The LiE code is given in Appendix~\ref{sec:lie_code}.

\section{\label{sec:proof pure}Proof of~\ref{main pure radical}: that some representations are $w_0$-pure}

We must prove that representations of highest weight $\lambda=kp_i\varpi_i$, $k\leq m_i$ are $w_0$-pure of sign $\sigma_i^k$ (with data $p_i$, $m_i$, $\sigma_i$ given in Table~\ref{mui-ni-table}).
We denote by~$\square$ the defining representation of each classical group ($\CC^{n+1}$~for~$A_n$, $\CC^{2n+1}$~for~$B_n$, $\CC^{2n}$~for $C_n$ and~$D_n$),
and introduce a basis of it:
for every $\eps \in \{-1, 0, 1\}$ and $i$ such that $\eps e_i$ (or for $A_n$ its orthogonal projection onto $\lie{h}^*$) is a weight of~$\square$, we call $h_{\eps i}$ some nonzero vector in the corresponding weight space.

For exceptional groups, all $m_i$ are finite so the algorithm of Section~\ref{sec:algo} suffices. We also use it for the representations with highest weight $2\varpi_3$ of $C_3$ and $2\varpi_4$ of~$C_4$. The results of these computations are given in the appendix, in Table~\ref{pure_by_hand}.

Most other cases are subrepresentations of $S^m\square$ of $A_n$ or $D_{2n+1}$, or one of $S^m\square$ or $\Lambda^m\square$ or $S^2(\Lambda^2\square)$ of $B_n$ or $C_n$ or $D_{2n}$, all of which will prove to be $w_0$-pure.
Here $S^m \rho$ and $\Lambda^m \rho$ denote the symmetric and the antisymmetric tensor powers of a representation~$\rho$.
The remaining cases are mapped to these by the isomorphisms $B_2\simeq C_2$ and $A_3\simeq D_3$ and the outer automorphisms $\ZZ/2\ZZ$ of $A_n$ and $\mathfrak{S}_3$ of~$D_4$.

For $A_n=\lie{sl}(n+1,\CC)$ the defining representation is $\square=\CC^{n+1}=\Span\{h_1,\dots,h_{n+1}\}$.
A representative $\overline{w_0}\in \SL(n+1,\CC)$ of $w_0$ acts on $\square$ by $h_j\mapsto h_{n+2-j}$ for $1\leq j<n+1$ and by $h_{n+1}\mapsto \sigma_1 h_1$ where $\sigma_1=(-1)^{\lfloor (n+1)/2\rfloor}$,
the sign being such that $\det\overline{w_0}=+1$.
We consider the representation $S^{k(n+1)}\square$.  Its zero-weight space~$V^0$ is spanned by symmetrized tensor products $h_{j_1}\otimes\dots\otimes h_{j_{k(n+1)}}$ in which each $h_j$ appears equally many times, namely~$k$ times.  Hence, $V^0$ is one-dimensional (the representation is thus $w_0$-pure) and spanned by the symmetrization of $v=h_1^{\otimes k}\otimes h_2^{\otimes k}\otimes\dots\otimes h_{n+1}^{\otimes k}$.  We compute $\overline{w_0}\cdot v = h_{n+1}^{\otimes k}\otimes\dots\otimes h_2^{\otimes k}\otimes(\sigma_1 h_1)^{\otimes k}$, whose symmetrization is equal to $\sigma_1^k$ times that of~$v$; this gives the announced sign $\sigma_1^k$.

For $D_{2n+1}=\lie{so}(4n+2,\CC)$ the defining representation is $\square=\CC^{4n+2}=\Span\{h_{\pm j}\mid 1\leq j\leq 2n+1\}$ and $\overline{w_0}$ maps $h_{\pm j}\mapsto h_{\mp j}$ for $1\leq j\leq 2n$ but fixes~$h_{\pm(2n+1)}$.  The zero-weight space $V^0$ of $S^{2k}\square$ is spanned by symmetrizations of $h_{j_1}\otimes h_{-j_1}\otimes\dots\otimes h_{j_k}\otimes h_{-j_k}$, each of which is fixed by~$\overline{w_0}$.  The representation is $w_0$-pure with $\sigma_1=+1$ as announced.

The cases of $B_n=\lie{so}(2n+1,\CC)$, $C_n=\lie{sp}(2n,\CC)$ and $D_{n\text{ even}}=\lie{so}(2n,\CC)$ are treated together.
\begin{itemize}
\item $B_n$ has $\square=\CC^{2n+1}=\Span\{h_j\mid -n\leq j\leq n\}$ and $\overline{w_0}$ acts by $h_j\mapsto h_{-j}$ for $j\neq 0$ and $h_0\mapsto (-1)^n h_0$.
\item $C_n$ has $\square=\CC^{2n}=\Span\{h_{\pm j}\mid 1\leq j\leq n\}$ and $\overline{w_0}$ acts by $h_j\mapsto h_{-j}$ and $h_{-j}\mapsto -h_j$ for $j>0$.
\item $D_n$ has $\square=\CC^{2n}=\Span\{h_{\pm j}\mid 1\leq j\leq n\}$ and, for $n$~even, $\overline{w_0}$ acts by $h_j\mapsto h_{-j}$ for all~$j$.
\end{itemize}
First consider $\Lambda^m\square$ and $S^m\square$.
Their zero-weight spaces are spanned by (anti)symmetrizations of $h_{j_1}\otimes h_{-j_1}\otimes\dots\otimes h_{j_k}\otimes h_{-j_k}\otimes h_0^{\otimes l}$, where $2k+l=m$.  Each of these vectors is fixed by $\overline{w_0}$ up to a sign that only depends on the group, the representation, and on $(k,l)$ or equivalently $(l,m)$.  For $C_n$ and $D_n$ we have $l=0$ so for each $m$ the representation is $w_0$-pure, with a sign $(-1)^k$ for $S^{2k}\square$ of $C_n$ and $\Lambda^{2k}\square$ of $D_n$, and no sign otherwise.  For $\Lambda^m\square$ of~$B_n$ we note that $l\in\{0,1\}$ is fixed by the parity of~$m$ so the representation is $w_0$-pure; its sign is $(-1)^{nl+k}=(-1)^{nm+\lfloor m/2\rfloor}=\sigma_m$.  For $S^m\square$ of~$B_n$ only the parity of~$l$ is fixed but the sign $(-1)^{nl}=(-1)^{nm}=\sigma_1^m$ still only depends on the representation; it confirms the data of Table~\ref{mui-ni-table}.
Finally consider the representation $S^2(\Lambda^2\square)$.  Its zero-weight space is spanned by symmetrizations of $(h_j\wedge h_{-j})\otimes(h_k\wedge h_{-k})$ and $(h_j\wedge h_k)\otimes(h_{-j}\wedge h_{-k})$ all of which are fixed by $\overline{w_0}$.

\section{\label{sec:ideal}Cartan product: $w_0$-mixed representations form an ideal}
Let $G$ be a simply-connected simple complex Lie group and $N$ a maximal unipotent subgroup of~$G$.
Define $\CC[G/N]$ the space of regular (\ie polynomial) functions on $G/N$.
Pointwise multiplication of functions is $G$-equivariant and makes $\CC[G/N]$ into a $\CC$-algebra without zero divisors (because $G/N$ is irreducible as an algebraic variety).

\begin{theorem}[{\cite[(3.20)--(3.21)]{VinPopBook}}]\label{thm:Vinberg}
  Each finite-dimensional representation of~$G$ (or equivalently of its Lie algebra~$\lie{g}$) occurs exactly once as a direct summand of the representation $\CC[G/N]$.
  The $\CC$-algebra $\CC[G/N]$ is graded in two ways:
  \begin{itemize}
  \item by the highest weight $\lambda$, in the sense that the product of a vector in $V_\lambda$ by a vector in $V_\mu$ lies in $V_{\lambda+\mu}$ (where $V_\lambda$ stands here for the subrepresentation of $\CC[G/N]$ with highest weight $\lambda$);
  \item by the actual weight $\lambda$, in the sense that the product of a weight vector with weight $\lambda$ by a weight vector with weight $\mu$ is still a weight vector, with weight $\lambda+\mu$.
  \end{itemize}
\end{theorem}

For given $\lambda$ and $\mu$, we call \emph{Cartan product} the induced bilinear map $\odot: V_\lambda \times V_\mu\to V_{\lambda+\mu}$.
Given $u\in V_\lambda$ and $v\in V_\mu$, this defines $u\odot v\in V_{\lambda+\mu}$ as the projection of $u\otimes v\in V_\lambda\otimes V_\mu=V_{\lambda+\mu}\oplus\dots$.
Since $\CC[G/N]$ has no zero divisor, $u\odot v\neq 0$ whenever $u\neq 0$ and $v\neq 0$.  We deduce the following.

\begin{lemma}\label{lem:ideal}
  The set of highest weights of $w_0$-mixed irreducible representations of~$\lie{g}$ is an ideal~$\mathcal{I}_{\lie{g}}$ of the additive monoid~$\mathcal{M}$ of dominant elements of the root lattice.
\end{lemma}

\begin{proof}
  Consider a $w_0$-mixed representation~$V_\lambda$ and a representation~$V_\mu$ whose highest weight is radical.
  We can choose $u_+$ and $u_-$ in the zero-weight space of $V_\lambda$ such that $w_0\cdot u_+=u_+$ and $w_0\cdot u_-=-u_-$, and choose $v$ in the zero-weight space of $V_\mu$ such that $w_0\cdot v=\pm v$ for some sign.
  Then $u_+\odot v$ and $u_-\odot v$ are non-zero elements of the zero-weight space of $V_{\lambda+\mu}$ on which $w_0$ acts by opposite signs.
\end{proof}


\section{\label{sec:inductionrule}Proof of~\ref{main mixed}: that other representations are $w_0$-mixed}

Let $\mathcal{I}^{\text{Table}}_{\lie{g}}$ be the set of dominant radical weights that are not of the form $\lambda=kp_i\varpi_i$, $k\leq m_i$ (with data $p_i$, $m_i$ given in Table~\ref{mui-ni-table}).  Observe that $\mathcal{I}^{\text{Table}}_{\lie{g}}$ is an ideal of~$\mathcal{M}$.
In Section~\ref{sec:proof pure} we showed $\mathcal{I}_{\lie{g}}\subset\mathcal{I}^{\text{Table}}_{\lie{g}}$.
We now show that $\mathcal{I}^{\text{Table}}_{\lie{g}}\subset\mathcal{I}_{\lie{g}}$, namely that $V_\lambda$ is $w_0$-mixed for radical~$\lambda$ other than those described by Table~\ref{mui-ni-table}.
By Lemma~\ref{lem:ideal}, it is enough to show this for the basis of $\mathcal{I}^{\text{Table}}_{\lie{g}}$.
For any given group, $\mathcal{I}^{\text{Table}}_{\lie{g}}$ has a finite basis so we simply use the algorithm of Section~\ref{sec:algo} to conclude for $A_{\leq 5}$, $B_{\leq 4}$, $C_{\leq 5}$, $D_{\leq 6}$ and all exceptional groups. The results of these computations are listed in the appendix, in Table~\ref{mixed_by_hand}.

Now let $\lie{g}$ be one of $A_{>5}$, $B_{>4}$, $C_{>5}$, $D_{>6}$ and $\lambda$ be in $\mathcal{I}^{\text{Table}}_{\lie{g}}$.
We proceed by induction on the rank of~$\lie{g}$.

Define as follows a reductive Lie subalgebra $\lie{f}\times\lie{g}'\subset\lie{g}$:
\begin{itemize}
\item If $\lie{g} = \lie{sl}(n,\CC)$, we choose $\lie{f}\times\lie{g}' \simeq \bigl(\lie{gl}(1,\CC)\times\lie{sl}(2,\CC)\bigr)\times\lie{sl}(n-2,\CC)$,
  where $\lie{f}$ has the roots $\pm (e_1-e_n)$ and $\lie{g'}$ has the roots $\pm(e_i-e_j)$ for $1<i<j<n$.
\item If $\lie{g} = \lie{so}(n,\CC)$, we choose $\lie{f}\times\lie{g}' \simeq \lie{so}(4,\CC)\times\lie{so}(n-4,\CC)$,
  where $\lie{f}$ has the roots $\pm e_1\pm e_2$ and $\lie{g'}$ has the roots $\pm e_i\pm e_j$ for $3\leq i<j\leq n$.
\item If $\lie{g} = \lie{sp}(2n,\CC)$, we choose $\lie{f}\times\lie{g}' \simeq \lie{sp}(2,\CC)\times\lie{sp}(2n-2,\CC)$,
  where $\lie{f}$ has the roots $\pm 2e_1$ and $\lie{g'}$ has the roots $\pm e_i\pm e_j$ for $2\leq i<j\leq n$ and $\pm 2e_i$ for $2\leq i\leq n$.
\end{itemize}
In all three cases, $\lie{f}\times\lie{g}'$ and $\lie{g}$ share their Cartan subalgebra hence restricting a representation~$V$ of~$\lie{g}$ to $\lie{f}\times\lie{g}'$ does not change the zero-weight space~$V^0$.
Additionally, consider any connected Lie group~$G$ with Lie algebra~$\lie{g}$: then the $w_0$~elements of the connected subgroup of~$G$ with Lie algebra $\lie{f}\times \lie{g}'$ and of~$G$ itself coincide, or more precisely have a common representative in~$G$, because the Lie algebras have the same Lie subalgebra~$\lie{s}$ defined in Proposition~\ref{sl2_copies}.
It follows that a representation of~$\lie{g}$ is $w_0$-mixed if and only if its restriction to $\lie{f}\times\lie{g}'$ is.

Next, decompose $V_\lambda=\bigoplus_\iota(V_{\xi_\iota}\otimes V_{\mu_\iota})$ into irreducible representations of $\lie{f}\times\lie{g}'$, where $\xi_\iota$ and $\mu_\iota$ are dominant weights of $\lie{f}$ and~$\lie{g}'$, respectively.
Consider the subspace
\begin{equation}\label{lamdba0bullet}
  V_\lambda^{(0,\bullet)}\coloneqq\bigoplus_\iota(V_{\xi_\iota}^0\otimes V_{\mu_\iota})\subset V_\lambda
\end{equation}
fixed by the Cartan algebra of~$\lie{f}$.
It is a representation of~$\lie{g}'$ whose zero-weight subspace coincides with that of~$V_\lambda$.
The direct sum obviously restricts to radical~$\xi_\iota$, and $\dim V_{\xi_\iota}^0=1$ because we chose $\lie{f}$ to be a product of $\lie{sl}(2,\CC)$ and $\lie{gl}(1,\CC)$ factors.
Thus the $w_0$~element of $\lie{g}$ acts on $V_{\xi_\iota}^0\otimes V_{\mu_\iota}$ in the same way, up to a sign, as the $w_0$~element of $\lie{g}'$ acts on~$V_{\mu_\iota}$.
Lemma~\ref{lem:nu-exists} shows that $V_\lambda^{(0,\bullet)}$ has an irreducible subrepresentation $V_\nu$ such that $\nu\in\mathcal{I}^{\text{Table}}_{\lie{g}'}$.
By the induction hypothesis, $V_\nu$~is then $w_0$-mixed hence $w_0$ has both eigenvalues $\pm 1$ on the zero-weight space $V_\lambda^0\subset V_\lambda^{(0,\bullet)}$, namely $V_\lambda$ is $w_0$-mixed.

This concludes the proof of Theorem~\ref{main_theorem}.

There remains to state and prove two lemmas.
Let $\lie{g}$ be $A_{n-1}$, $B_n$, $C_n$ or~$D_n$ and let $\lambda$ be a dominant radical weight of~$\lie{g}$.
It can then be expressed in the standard basis $e_1,\dots,e_n$ as $\lambda=\sum_{i=1}^n\lambda_ie_i$ where $\lambda_1\geq\lambda_2\geq\dots\geq\lambda_n$ are integers subject to: for $A_{n-1}$, $\sum_i\lambda_i=0$; for $B_n$, $\lambda_n\geq 0$; for $C_n$, $\lambda_n\geq 0$ and $\sum_i\lambda_i\in 2\ZZ$; for $D_n$, $\lambda_{n-1}\geq\abs{\lambda_n}$ and $\sum_i\lambda_i\in 2\ZZ$.
In addition, let $\lie{f}\times\lie{g}'\subset\lie{g}$ be the subalgebra defined above.
We identify weights of~$\lie{g}'$ with the corresponding weights of~$\lie{g}$ (acting trivially on the Cartan subalgebra of~$\lie{f}$). Note that this introduces a shift in their coordinates: the dual of the Cartan subalgebra of~$\lie{g}'$ is spanned by a subset of the vectors~$e_i$ (corresponding to~$\lie{g}$) that starts at $e_2$ or $e_3$, not at $e_1$ as expected.

\begin{lemma}\label{lem:define-mu}
  Let $\mu$ be the dominant weight of $\lie{g}'$ defined as follows.
  \begin{itemize}
  \item For $A_{n-1}$, $\mu=\bigl(\sum_{i=1}^{\ell-1}\lambda_ie_{i+1}\bigr)+\lambda_\ell e_\ell+\bigl(\sum_{i=\ell+1}^n\lambda_ie_{i-1}\bigr)$
    where $1<\ell<n$ is an index such that $\lambda_{\ell-1}+\lambda_{\ell}\geq 0\geq\lambda_{\ell}+\lambda_{\ell+1}$ (when several~$\ell$ obey this, $\mu$ does not depend on the choice).
  \item For $B_n$, $\mu=\sum_{i=1}^{n-2}\lambda_i e_{i+2}$.
  \item For $C_n$, $\mu=\sum_{i=1}^{n-1}\lambda_i e_{i+1}-\eta e_n$ where $\eta\in\{0,1\}$ obeys $\eta\equiv\lambda_n\pmod{2}$.
  \item For $D_n$, $\mu=\sum_{i=1}^{n-2}\lambda_i e_{i+2}-\eta e_n$ where $\eta\in\{0,1\}$ obeys $\eta\equiv\lambda_{n+1}+\lambda_n\pmod{2}$.
  \end{itemize}
  Then $V_\mu$ is a sub-representation of the space $V_\lambda^{(0,\bullet)}$ defined earlier.
\end{lemma}

\begin{proof}[for $A_{n-1}$]
  Let $\nu=\sum_{i=2}^{n-1}\nu_ie_i$ be a dominant radical weight of~$\lie{g}'$.
  The weight~$\nu$ is among weights of~$V_\lambda^{(0,\bullet)}$ if and only if it is among weights of~$V_\lambda$.
  The condition is that $\langle\lambda-\tilde{\nu},\varpi_k\rangle\geq 0$ for all~$k$,
  where $\tilde{\nu}$ is the unique dominant weight of~$\lie{g}$ in the orbit of~$\nu$ under the Weyl group of~$\lie{g}$.

  Explicitly, $\tilde{\nu}=\bigl(\sum_{i=1}^{p-1}\nu_{i+1}e_i\bigr)+\sum_{i=p+2}^n\nu_{i-1}e_i$ where $p$ is any index such that $\nu_p\geq 0\geq\nu_{p+1}$.
  Then the condition is $\sum_{i=1}^k\lambda_i\geq\sum_{i=2}^{k+1}\nu_i$ for $1\leq k<p$ and
  $\sum_{i=1}^p\lambda_i\geq\sum_{i=2}^p\nu_i$ and
  $\sum_{i=1}^k\lambda_i\geq\sum_{i=2}^{k-1}\nu_i$ for $p<k<n$.
  Let us show that this is equivalent to
  \begin{equation}\label{A-type-nu-condition}
    \sum_{i=2}^k\nu_i\leq\min\biggl(\sum_{i=1}^{k-1}\lambda_i,\sum_{i=1}^{k+1}\lambda_i\biggr) \text{ for all } 2\leq k\leq n-2 .
  \end{equation}
  In one direction the only non-trivial statement is that $2 \sum_{i=1}^p\lambda_i \geq \sum_{i=1}^{p-1}\lambda_i + \sum_{i=1}^{p+1} \lambda_i \geq 2 \sum_{i=2}^p\nu_i$, where we used $2\lambda_p\geq \lambda_p+\lambda_{p+1}$.  In the other direction we check $\sum_{i=2}^k\nu_i\leq\sum_{i=2}^{\min(p,k+2)}\nu_i\leq\sum_{i=1}^{k+1}\lambda_i$ for $k\leq p-1$ using $\nu_2\geq\cdots\geq\nu_p\geq 0$ and similarly for $p+1\leq k$ using $0\geq\nu_{p+1}\geq\dots\geq\nu_{n-1}$.

  Now, $\lambda_{\ell-1}+\lambda_{\ell}\geq 0\geq\lambda_{\ell}+\lambda_{\ell+1}$ implies $\lambda_{\ell-2}\geq\lambda_{\ell-1}\geq\lambda_{\ell-1}+\lambda_{\ell}+\lambda_{\ell+1}\geq\lambda_{\ell+1}\geq\lambda_{\ell+2}$, so $\mu$ is a dominant weight of~$\lie{g}'$.  It is radical because $\sum_{i=2}^{n-1}\mu_i=\sum_{i=1}^n\lambda_i=0$.
  Furthermore, $\mu$ saturates all bounds~\eqref{A-type-nu-condition} (with $\nu$ replaced by~$\mu$), as seen using $\lambda_k+\lambda_{k+1}\geq 0$ or $\leq 0$ for $k<\ell$ or $k\geq\ell$ respectively.
  In particular we deduce that $\mu$~is among the weights of $V_\lambda^{(0,\bullet)}$, hence of some irreducible summand $V_\nu\subset V_\lambda^{(0,\bullet)}$.
  The dominant radical weight~$\nu$ of $\lie{g}'$ must also obey~\eqref{A-type-nu-condition}, namely $\sum_{i=2}^k\nu_i\leq\sum_{i=2}^k\mu_i$ (due to the aforementioned saturation).
  Since $\mu$ is dominant and among weights of~$V_\nu$, we must also have $\langle\nu-\mu,\varpi_k'\rangle\geq 0$ for all fundamental weights $\varpi_k'$ of~$\lie{g}'$.
  This is precisely the reverse inequality $\sum_{i=2}^k\nu_i\geq\sum_{i=2}^k\mu_i$.  We conclude that $\mu=\nu$.
\end{proof}
\begin{proof}[for $B_n$, $C_n$, $D_n$]
  Let $\eps=1$ for $C_n$ and otherwise $\eps=2$.
  Again, a dominant radical weight $\nu=\sum_{i=1+\eps}^n(\nu_ie_i)$ of $\lie{g}'$ is a weight of $V_\lambda^{(0,\bullet)}$ if and only if all $\langle\lambda-\tilde{\nu},\varpi_k\rangle\geq 0$, where $\tilde{\nu}$ is the unique dominant weight of~$\lie{g}$ in the Weyl orbit of~$\nu$.
  In all three cases, $\tilde{\nu}=\sum_{i=1}^{n-\eps}\abs{\nu_{i+\eps}}e_i$, where the absolute value is only useful for the $\nu_n$ component for~$D_n$.
  The condition is worked out to be $\sum_{i=1}^k\lambda_i\geq\sum_{i=1}^k\abs{\nu_{i+\eps}}$ for $1\leq k\leq n-\eps$.
  It is easy to check that $\mu$ is a dominant radical weight of~$\lie{g}'$ and it obeys these conditions.

  Consider now an irreducible summand $V_\nu\subset V_\lambda^{(0,\bullet)}$ that has~$\mu$ among its weights.
  On the one hand, $\sum_{i=1}^k\lambda_i\geq\sum_{i=1}^k\abs{\nu_{i+\eps}}$ for $1\leq k\leq n-\eps$, where the absolute value is only useful for~$\nu_n$ for~$D_n$.
  On the other hand, $\langle\nu-\mu,\varpi'\rangle\geq 0$ for all dominant weights~$\varpi'$ of~$\lie{g}'$ (in particular $e_{1+\eps}+\dots+e_{k+\eps}$),
  so $\sum_{i=1}^k\nu_{i+\eps}\geq\sum_{i=1}^k\mu_{i+\eps}$ for $1\leq k\leq n-\eps$.
  The two inequalities fix $\nu_i=\mu_i$ for all~$i$, except $i=n$ when $\eta=1$ for $C_n$ and~$D_n$:
  in these cases we conclude by using $\sum_i\nu_i-\sum_i\mu_i\in 2\ZZ$ since both weights are radical.
\end{proof}

\begin{lemma}\label{lem:nu-exists}
  For any $\lambda\in\mathcal{I}^{\text{Table}}_{\lie{g}}$, there exists $\nu\in\mathcal{I}^{\text{Table}}_{\lie{g}'}$ such that the representation of~$\lie{g}'$ with highest weight~$\nu$ is a subrepresentation of $V_\lambda^{(\bullet,0)}$.
\end{lemma}
\begin{proof}[for $A_{n-1}$ with $n\geq 7$]
  If the weight~$\mu$ defined by Lemma~\ref{lem:define-mu} is in $\mathcal{I}^{\text{Table}}_{\lie{g}'}$ we are done.
  Otherwise, $\mu=m(n-2)\varpi_1'$ or $\mu=m(n-2)\varpi_{n-3}'$.
  By symmetry under $e_i\mapsto -e_{n+1-i}$ it is enough to consider the second case, so $\mu=\sum_{i=2}^{n-1}\mu_ie_i$ with
  $\mu_i=m$ for $2\leq i\leq n-2$ and $\mu_{n-1}=-m(n-3)$.
  By the construction of $\mu$ in terms of~$\lambda$ we know that there exists $1<\ell<n$ such that
  $\mu_i=\lambda_{i-1}\geq 0$ for $1<i<\ell$ and $\lambda_{\ell-1}\geq\mu_\ell=\lambda_{\ell-1}+\lambda_{\ell}+\lambda_{\ell+1}\geq\lambda_{\ell+1}$ and $\mu_i=\lambda_{i+1}\leq 0$ for $\ell<i<n$.
  Since only $\mu_{n-1}\leq 0$, the last constraint sets $\ell=n-2$ or $\ell=n-1$.
  In the first case, we learn that $\lambda_i=m$ for $1\leq i\leq n-4$, but also that $m=\mu_{n-3}=\lambda_{n-4}\geq\lambda_{n-3}\geq\mu_{n-2}=m$ so $\lambda_{n-3}=m$,
  thus $\lambda_{n-2}+\lambda_{n-1}=\mu_{n-2}-\lambda_{n-3}=0$ and we can change $\ell$ to $n-1$
  (recall that the choice of $\ell$ such that $\lambda_{\ell-1}+\lambda_{\ell}\geq 0\geq \lambda_{\ell}+\lambda_{\ell+1}$ does not affect~$\mu$).
  We are thus left with the case $\ell=n-1$, where $\lambda_i=m$ for $1\leq i\leq n-3$, and where $\lambda_{n-2}+\lambda_{n-1}\geq 0$ and $m=\lambda_{n-3}\geq\lambda_{n-2}$.

  We conclude that $\lambda=m\bigl(\sum_{i=1}^{n-3}e_i\bigr)+le_{n-2}+ke_{n-1}-\bigl((n-3)m+l+k\bigr)e_n$ for integers $m\geq l\geq\abs{k}$,
  with the exclusion of the case $k=l=m$ because of $\lambda\in\mathcal{I}^{\text{Table}}_{\lie{g}}$.
  For these dominant weights, the particular irreducible summand $V_\mu\subset V_\lambda^{(0,\bullet)}$ of Lemma~\ref{lem:define-mu} is $w_0$-pure, but we now determine another summand that is $w_0$-mixed.
  The branching rules from $\lie{g}$ to $\lie{f}\times\lie{g}'$ can easily be deduced from the classical branching rules from $\lie{gl}(n,\CC)$ to $\lie{gl}(n-1,\CC)$ (given for example in \cite[Theorem 9.14]{Kna96}).
  Namely, consider the representation of $\lie{gl}(n,\CC)$ on $V_\lambda$ such that the diagonal $\lie{gl}(1,\CC)$ acts by zero.
  Then $V_\lambda^{(0,\bullet)}\subset V_\lambda$ is the subspace on which all three $\lie{gl}(1,\CC)$ factors of $\lie{gl}(1,\CC)\times\lie{gl}(n-2,\CC)\times\lie{gl}(1,\CC)\subset\lie{gl}(n,\CC)$ act by zero.
  It decomposes into irreducible representations of $\lie{g}'\simeq\lie{sl}(n-2,\CC)$ with highest weights $\lambda''=\sum_{i=2}^{n-1}\lambda''_ie_i$ such that $\sum_i\lambda''_i=0$ and such that there exists $\lambda'_1,\dots,\lambda'_{n-1}$ with $\sum_i\lambda'_i=0$, and $\lambda_1\geq\lambda'_1\geq\lambda_2\geq\dots\geq\lambda'_{n-1}\geq\lambda_n$ and $\lambda'_1\geq\lambda''_2\geq\lambda'_2\geq\dots\geq\lambda''_{n-1}\geq\lambda'_{n-1}$.
  Concretely we focus on the summand where $(\lambda_i)_{i=1}^n$ and $(\lambda'_i)_{i=1}^{n-1}$ and $(\lambda''_i)_{i=2}^{n-1}$ all take the form $(m,\dots,m,l,k,-S)$ where $S$ is the sum of all other entries, with a different number of~$m$ in each case.  Given that we started in rank at least~$6$, the resulting weight $\lambda''$ cannot be a multiple of a fundamental weight, hence $\lambda''\in\mathcal{I}^{\text{Table}}_{\lie{g}'}$.
\end{proof}

\begin{proof}[for $B_n$ with $n\geq 5$, $C_n$ with $n\geq 6$, $D_n$ with $n\geq 7$]
  We recall $\eps=1$ for $C_n$ and otherwise $\eps=2$.
  If the weight~$\mu$ defined by Lemma~\ref{lem:define-mu} is in $\mathcal{I}^{\text{Table}}_{\lie{g}'}$ we are done.
  Otherwise, $\mu$ can take a few possible forms because we took $\rank\lie{g}'=n-\eps$ large enough to avoid special values listed in Table~\ref{mui-ni-table}.
  Note that by construction of $\mu=\sum_{i=1+\eps}^n\mu_ie_i$ we have $\lambda_i=\mu_{i+\eps}$ for $1\leq i\leq n-3$ for $D_n$ and $1\leq i\leq n-2$ for $B_n$ and~$C_n$.
  The possible dominant radical weights not in~$\mathcal{I}^{\text{Table}}_{\lie{g}'}$ are as follows.
  \begin{itemize}
  \item First, $\mu=m\varpi_1'=me_{1+\eps}$, where additionally $m$ is even for $C_n$ and $D_n$.
    Then $\lambda_1=\mu_{1+\eps}=m$ and $\lambda_2=\mu_{2+\eps}=0$ fix $\lambda=m\varpi_1$, which is not in $\mathcal{I}^{\text{Table}}_{\lie{g}}$.
  \item Second, $\mu=2\varpi_2'=2(e_{1+\eps}+e_{2+\eps})$, except for $D_n$ with odd~$n$.
    Then $\lambda_1=\lambda_2=2$ and $\lambda_3=0$ fix $\lambda=2\varpi_2$, which is not in $\mathcal{I}^{\text{Table}}_{\lie{g}}$.
  \item Third, $\mu=\sum_{i=1}^m e_{i+\eps}$ for some $m\geq 2$, except for $D_n$ with odd~$n$, and where additionally $m$ is even for $D_n$ with even~$n$ and for~$C_n$.
    Since $\lambda_1=\mu_{1+\eps}=1$ and $\lambda$ is dominant we deduce that either $\lambda_1=\dots=\lambda_p=1$ for some $p$ and all other $\lambda_i=0$, or (only in the $D_n$ case) $\lambda_1=\dots=\lambda_{n-1}=1=-\lambda_n$.  These weights~$\lambda$ are not in $\mathcal{I}^{\text{Table}}_{\lie{g}}$.
    Note of course that $p$ and $m$ are not independent; for example for $m\leq n-3$ one has $m=p$.
  \item Fourth, $\mu=\bigl(\sum_{i=1}^{n-3}e_{i+2}\bigr)-e_n$ for $D_n$ with even~$n$.  This weight is not of the form of Lemma~\ref{lem:define-mu} because one would need $-1=\lambda_{n-2}-\eta\geq-\eta\geq-1$ hence $\eta=1$ and $\lambda_{n-2}=0$, so $\lambda_{n-1}=\lambda_n=0$ so $1=\eta\equiv\lambda_{n-1}+\lambda_n=0\pmod{2}$.
    \qedhere
  \end{itemize}
\end{proof}

\appendix

\section{Tables of values computed by hand}
\label{sec:by_hand}

To finish the proof, it remains to compute the $w_0$-signatures of the 167 representations listed in the two tables below. (Actually, by taking advantage of the outer automorphisms, we can reduce the number of computations to 138.) We did the computations with the algorithm described in Section~\ref{sec:algo} and implemented in Appendix~\ref{sec:lie_code}. The results are listed in the following two tables.

In both tables, each line gives the $w_0$-signature of a representation with highest weight~$\lambda$, encoded in two different ways : first by the coordinates of~$\lambda$ in the $(\varpi_1, \ldots, \varpi_r)$ basis (sometimes called the ``Dynkin coefficients'' of the representation), and also by the coordinates of~$\lambda$ in the $(e_1, \ldots, e_n)$ basis (called $\lambda_1, \ldots, \lambda_n$ in Section~\ref{sec:inductionrule}). For reference, the dimension of the representation is also given.

For each Lie algebra~$\mathfrak{g}$, the representations are usually listed in the lexicographic order of their Dynkin coefficients. However, representations that map to each other under some outer automorphism of~$\mathfrak{g}$, and consequently automatically have the same dimension and $w_0$-signature, are listed together. In this case, only one representative of the orbit (the first one in lexicographic order) is given a full entry; the other ones are written immediately below it, in gray and with ``same as above'' comment over the last two columns.

The representations covered in Table~\ref{pure_by_hand} are all the supposedly pure representations of exceptional simple Lie algebras (whose list can be immediately read off Table~\ref{mui-ni-table}), plus the representation of $C_3$ with highest weight $2\varpi_3$ and of $C_4$ with highest weight $2\varpi_4$.

\begin{table}[h]
  \caption{\label{root_lattice_congruences}Linear congruences that the coordinates $c_1, \ldots, c_r$ must satisfy in order for the weight $\lambda = \sum_{i=1}^r c_i \varpi_i$ to lie in the root lattice $\Lambda$.}
  \centering\bigskip
\begin{tabular}{c>{\hspace{1em}}l}
$\mathfrak{g}$ & Condition \\[1ex] \hline
$A_n$ & $\sum_{i=1}^n i c_i \equiv 0 \pmod{n+1}$ \\
$B_n$ & $c_n \equiv 0 \pmod{2}$ \\[1ex]
$C_n$ & $\sum_{i=0}^{\lfloor \frac{n-1}{2} \rfloor} c_{2i+1} \equiv 0 \pmod{2}$ \\[2ex]
$D_{2n}$ & $\begin{cases} \left(\sum_{i=0}^{n-2} c_{2i+1}\right) + c_{2n-1} \equiv 0 \pmod{2} \\
                          \left(\sum_{i=0}^{n-2} c_{2i+1}\right) + c_{2n} \equiv 0 \pmod{2} \end{cases}$ \\[3ex]
$D_{2n+1}$ & $2 \left(\sum_{i=0}^{n-1} c_{2i+1}\right) + c_{2n} - c_{2n+1} \equiv 0 \pmod{4}$ \\
$E_6$ & $c_1 - c_3 + c_5 - c_6 \equiv 0 \pmod{3}$ \\
$E_7$ & $c_2 + c_5 + c_7 \equiv 0 \pmod{2}$ \\
$E_8$, $F_4$, $G_2$ & $0 = 0$ (\ie the root lattice is the whole weight lattice) \\ \hline
\end{tabular}
\end{table}

\begin{table}
\caption{\label{pure_by_hand}Representations whose $w_0$-purity is checked by explicit computation.}
\centering\bigskip
\begin{tabular}{l>{\hspace{1em}}c>{\hspace{1em}}c>{\hspace{1em}}r>{\hspace{1em}}c}
$\mathfrak{g}$ & \multicolumn{1}{>{\hspace{1em}}c}{\parbox{2.2cm}{\centering Coordinates of $\lambda$\\ in $(\varpi_i)$ basis}} & \multicolumn{1}{>{\hspace{1em}}c}{\parbox{2.2cm}{\centering Coordinates of $\lambda$\\ in $(e_i)$ basis}} & $\dim(V_\lambda)$ & \multicolumn{1}{>{\hspace{1em}}c}{\parbox{1.7cm}{\centering $w_0$-signature\\ of~$V_\lambda$}} \\ \hline

\multirow{4}{1em}{$E_7$}	& $(0,0,0,0,0,0,2)$	& $(0,0,0,0,0,2,-1,1)$	& 1463	& $(0,21)$ \\*
	& $(0,0,0,0,0,1,0)$	& $(0,0,0,0,1,1,-1,1)$	& 1539	& $(27,0)$ \\
	& $(1,0,0,0,0,0,0)$	& $(0,0,0,0,0,0,-1,1)$	& 133	& $(0,7)$ \\*
	& $(2,0,0,0,0,0,0)$	& $(0,0,0,0,0,0,-2,2)$	& 7371	& $(63,0)$ \\ \hline

\multirow{3}{1em}{$E_8$}	& $(0,0,0,0,0,0,0,1)$	& $(0,0,0,0,0,0,1,1)$	& 248	& $(0,8)$ \\*
	& $(0,0,0,0,0,0,0,2)$	& $(0,0,0,0,0,0,2,2)$	& 27000	& $(120,0)$ \\*
	& $(1,0,0,0,0,0,0,0)$	& $(0,0,0,0,0,0,0,2)$	& 3875	& $(35,0)$ \\ \hline

\multirow{4}{1em}{$F_4$}	& $(0,0,0,1)$	& $(1,0,0,0)$	& 26	& $(2,0)$ \\*
	& $(0,0,0,2)$	& $(2,0,0,0)$	& 324	& $(12,0)$ \\
	& $(1,0,0,0)$	& $(1,1,0,0)$	& 52	& $(0,4)$ \\*
	& $(2,0,0,0)$	& $(2,2,0,0)$	& 1053	& $(21,0)$ \\ \hline

\multirow{4}{1em}{$G_2$}	& $(0,1)$	& $(-1,-1,2)$	& 14	& $(0,2)$ \\*
	& $(0,2)$	& $(-2,-2,4)$	& 77	& $(5,0)$ \\
	& $(1,0)$	& $(0,-1,1)$	& 7	& $(0,1)$ \\*
	& $(2,0)$	& $(0,-2,2)$	& 27	& $(3,0)$ \\ \hline
	
$C_3$	& $(0,0,2)$	& $(2,2,2)$	& 84	& $(0,4)$ \\ \hline

$C_4$	& $(0,0,0,2)$	& $(2,2,2,2)$	& 594	& $(10,0)$ \\ \hline
\end{tabular}
\end{table}

The representations covered in Table~\ref{mixed_by_hand} are those that form a basis of $\mathcal{I}^{\text{Table}}$ as an ideal of~$\mathcal{M}$, for the algebras $A_{\leq 5}$, $B_{\leq 4}$, $C_{\leq 5}$, $D_{\leq 6}$ and the exceptional algebras. Recall that for any $\mathfrak{g}$, $\mathcal{M}$ is the set of vectors $\lambda = \sum_{i=1}^r c_i \varpi_i$ such that:
\begin{enumerate}
\item $\forall i = 1, \ldots r,\quad c_i \in \mathbb{Z}$ (this is the definition of the weight lattice);
\item $\forall i = 1, \ldots r,\quad c_i \geq 0$ (this is the dominance condition);
\item $\lambda$ is in the root lattice~$\Lambda$, which is a sublattice of the weight lattice; so this condition is equivalent to a system of linear congruences. These congruences are well-known, and given in Table~\ref{root_lattice_congruences}.
\end{enumerate}

Now $\mathcal{I}^{\text{Table}}$ is obtained from $\mathcal{M}$ by removing all the weights of the form $k p_i \varpi_i$ with $k \leq m_i$. Given all of this, computation of its basis becomes straightforward.

\clearpage
\begin{small}
\begin{longtable}{l>{\hspace{1em}}c>{\hspace{1em}}c>{\hspace{1em}}r>{\hspace{1em}}c}
\caption{\label{mixed_by_hand}Representations whose $w_0$-mixedness is checked by explicit computation.} \\
$\mathfrak{g}$ & \multicolumn{1}{>{\hspace{1em}}c}{\parbox{2.2cm}{\centering Coordinates of $\lambda$\\ in $(\varpi_i)$ basis}} & \multicolumn{1}{>{\hspace{1em}}c}{\parbox{2.2cm}{\centering Coordinates of $\lambda$\\ in $(e_i)$ basis}} & $\dim(V_\lambda)$ & \multicolumn{1}{>{\hspace{1em}}c}{\parbox{1.7cm}{\centering $w_0$-signature\\ of~$V_\lambda$}} \\ \hline
\endfirsthead
\caption[]{Representations whose $w_0$-mixedness is checked by explicit computation (continued).} \\
$\mathfrak{g}$ & \multicolumn{1}{>{\hspace{1em}}c}{\parbox{2.2cm}{\centering Coordinates of $\lambda$\\ in $(\varpi_i)$ basis}} & \multicolumn{1}{>{\hspace{1em}}c}{\parbox{2.2cm}{\centering Coordinates of $\lambda$\\ in $(e_i)$ basis}} & $\dim(V_\lambda)$ & \multicolumn{1}{>{\hspace{1em}}c}{\parbox{1.7cm}{\centering $w_0$-signature\\ of~$V_\lambda$}} \\ \hline
\endhead
$A_2$	& $(1,1)$	& $(1,0,-1)$	& 8	& $(1,1)$ \\* \hline

\topic{$A_3$}	& $(0,1,2)$	& $(1,1,0,-2)$	& 45	& $(1,2)$ \\*
\lasttopic	& \faded $(2,1,0)$	& \faded $(2,0,-1,-1)$	& \multicolumn{2}{r}{\faded same as above \quad \qquad ~} \\*
\lasttopic	& $(1,0,1)$	& $(1,0,0,-1)$	& 15	& $(1,2)$ \\ \hline

\topic{$A_4$}	& $(0,0,1,3)$	& $(1,1,1,0,-3)$	& 224	& $(2,2)$ \\*
\lasttopic	& \faded $(3,1,0,0)$	& \faded $(3,0,-1,-1,-1)$	& \multicolumn{2}{r}{\faded same as above \quad \qquad ~} \\
\lasttopic	& $(0,0,2,1)$	& $(1,1,1,-1,-2)$	& 175	& $(3,2)$ \\*
\lasttopic	& \faded $(1,2,0,0)$	& \faded $(2,1,-1,-1,-1)$	& \multicolumn{2}{r}{\faded same as above \quad \qquad ~} \\
\lasttopic	& $(0,0,5,0)$	& $(2,2,2,-3,-3)$	& 1176	& $(2,4)$ \\*
\lasttopic	& \faded $(0,5,0,0)$	& \faded $(3,3,-2,-2,-2)$	& \multicolumn{2}{r}{\faded same as above \quad \qquad ~} \\ 
\lasttopic	& $(0,1,0,2)$	& $(1,1,0,0,-2)$	& 126	& $(2,4)$ \\*
\lasttopic	& \faded $(2,0,1,0)$	& \faded $(2,0,0,-1,-1)$	& \multicolumn{2}{r}{\faded same as above \quad \qquad ~} \\ 
\lasttopic	& $(0,1,1,0)$	& $(1,1,0,-1,-1)$	& 75	& $(3,2)$ \\
\lasttopic	& $(0,3,0,1)$	& $(2,2,-1,-1,-2)$	& 700	& $(4,6)$ \\*
\lasttopic	& \faded $(1,0,3,0)$	& \faded $(2,1,1,-2,-2)$	& \multicolumn{2}{r}{\faded same as above \quad \qquad ~} \\*
\lasttopic	& $(1,0,0,1)$	& $(1,0,0,0,-1)$	& 24	& $(2,2)$ \\ \hline

\topic{$A_5$}	& $(0,0,0,1,4)$	& $(1,1,1,1,0,-4)$	& 1050	& $(3,2)$ \\* 
\lasttopic	& \faded $(4,1,0,0,0)$	& \faded $(4,0,-1,-1,-1,-1)$	& \multicolumn{2}{r}{\faded same as above \quad \qquad ~} \\
\lasttopic	& $(0,0,0,2,2)$	& $(1,1,1,1,-1,-3)$	& 1134	& $(3,6)$ \\*
\lasttopic	& \faded $(2,2,0,0,0)$	& \faded $(3,1,-1,-1,-1,-1)$	& \multicolumn{2}{r}{\faded same as above \quad \qquad ~} \\
\lasttopic	& $(0,0,0,3,0)$	& $(1,1,1,1,-2,-2)$	& 490	& $(4,1)$ \\*
\lasttopic	& \faded $(0,3,0,0,0)$	& \faded $(2,2,-1,-1,-1,-1)$	& \multicolumn{2}{r}{\faded same as above \quad \qquad ~} \\ 
\lasttopic	& $(0,0,1,0,3)$	& $(1,1,1,0,0,-3)$	& 840	& $(6,4)$ \\*
\lasttopic	& \faded $(3,0,1,0,0)$	& \faded $(3,0,0,-1,-1,-1)$	& \multicolumn{2}{r}{\faded same as above \quad \qquad ~} \\ 
\lasttopic	& $(0,0,1,1,1)$	& $(1,1,1,0,-1,-2)$	& 896	& $(8,8)$ \\*
\lasttopic	& \faded $(1,1,1,0,0)$	& \faded $(2,1,0,-1,-1,-1)$	& \multicolumn{2}{r}{\faded same as above \quad \qquad ~} \\ 
\lasttopic	& $(0,0,2,0,0)$	& $(1,1,1,-1,-1,-1)$	& 175	& $(1,4)$ \\ 
\lasttopic	& $(0,1,0,0,2)$	& $(1,1,0,0,0,-2)$	& 280	& $(4,6)$ \\*
\lasttopic	& \faded $(2,0,0,1,0)$	& \faded $(2,0,0,0,-1,-1)$	& \multicolumn{2}{r}{\faded same as above \quad \qquad ~} \\ 
\lasttopic	& $(0,1,0,1,0)$	& $(1,1,0,0,-1,-1)$	& 189	& $(6,3)$ \\ 
\lasttopic	& $(0,2,1,0,1)$	& $(2,2,0,-1,-1,-2)$	& 5670	& $(21,24)$ \\* 
\lasttopic	& \faded $(1,0,1,2,0)$	& \faded $(2,1,1,0,-2,-2)$	& \multicolumn{2}{r}{\faded same as above \quad \qquad ~} \\*
\lasttopic	& $(1,0,0,0,1)$	& $(1,0,0,0,0,-1)$	& 35	& $(2,3)$ \\ \hline

$B_2$	& $(1,2)$	& $(2,1)$	& 35	& $(1,2)$ \\ \hline

\topic{$B_3$}	& $(0,0,4)$	& $(2,2,2)$	& 294	& $(7,1)$ \\*
\lasttopic	& $(0,1,2)$	& $(2,2,1)$	& 378	& $(3,9)$ \\ 
\lasttopic	& $(0,3,0)$	& $(3,3,0)$	& 825	& $(1,16)$ \\ 
\lasttopic	& $(1,0,2)$	& $(2,1,1)$	& 189	& $(3,6)$ \\*
\lasttopic	& $(1,1,0)$	& $(2,1,0)$	& 105	& $(3,2)$ \\ \hline

\topic{$B_4$}	& $(0,0,0,4)$	& $(2,2,2,2)$	& 2772	& $(24,4)$ \\*
\lasttopic	& $(0,0,1,2)$	& $(2,2,2,1)$	& 4158	& $(18,28)$ \\ 
\lasttopic	& $(0,0,2,0)$	& $(2,2,2,0)$	& 1980	& $(24,4)$ \\ 
\lasttopic	& $(0,1,0,2)$	& $(2,2,1,1)$	& 2772	& $(12,32)$ \\ 
\lasttopic	& $(0,1,1,0)$	& $(2,2,1,0)$	& 1650	& $(18,8)$ \\ 
\lasttopic	& $(0,3,0,0)$	& $(3,3,0,0)$	& 4004	& $(4,40)$ \\ 
\lasttopic	& $(1,0,0,2)$	& $(2,1,1,1)$	& 924	& $(12,8)$ \\ 
\lasttopic	& $(1,0,1,0)$	& $(2,1,1,0)$	& 594	& $(6,12)$ \\*
\lasttopic	& $(1,1,0,0)$	& $(2,1,0,0)$	& 231	& $(3,4)$ \\ \hline

\topic{$C_3$}	& $(0,0,4)$	& $(4,4,4)$	& 1001	& $(10,1)$ \\*
\lasttopic	& $(0,1,2)$	& $(3,3,2)$	& 594	& $(3,9)$ \\ 
\lasttopic	& $(0,3,0)$	& $(3,3,0)$	& 385	& $(10,1)$ \\ 
\lasttopic	& $(1,0,1)$	& $(2,1,1)$	& 70	& $(1,3)$ \\*
\lasttopic	& $(2,1,0)$	& $(3,1,0)$	& 189	& $(3,6)$ \\ \hline

\topic{$C_4$}	& $(0,0,0,3)$	& $(3,3,3,3)$	& 4719	& $(19,4)$ \\*
\lasttopic	& $(0,0,2,0)$	& $(2,2,2,0)$	& 825	& $(1,16)$ \\ 
\lasttopic	& $(0,1,0,1)$	& $(2,2,1,1)$	& 792	& $(12,4)$ \\ 
\lasttopic	& $(0,3,0,0)$	& $(3,3,0,0)$	& 2184	& $(28,4)$ \\ 
\lasttopic	& $(1,0,1,0)$	& $(2,1,1,0)$	& 315	& $(3,8)$ \\ 
\lasttopic	& $(2,0,0,1)$	& $(3,1,1,1)$	& 1155	& $(7,12)$ \\*
\lasttopic	& $(2,1,0,0)$	& $(3,1,0,0)$	& 594	& $(6,12)$ \\ \hline

\topic{$C_5$}	& $(0,0,0,0,2)$	& $(2,2,2,2,2)$	& 4719	& $(1,26)$ \\*
\lasttopic	& $(0,0,0,2,0)$	& $(2,2,2,2,0)$	& 7865	& $(50,5)$ \\ 
\lasttopic	& $(0,0,1,0,1)$	& $(2,2,2,1,1)$	& 8580	& $(20,40)$ \\ 
\lasttopic	& $(0,0,2,0,0)$	& $(2,2,2,0,0)$	& 4004	& $(4,40)$ \\ 
\lasttopic	& $(0,1,0,1,0)$	& $(2,2,1,1,0)$	& 5005	& $(40,15)$ \\ 
\lasttopic	& $(0,3,0,0,0)$	& $(3,3,0,0,0)$	& 8250	& $(60,10)$ \\ 
\lasttopic	& $(1,0,0,0,1)$	& $(2,1,1,1,1)$	& 1155	& $(5,10)$ \\ 
\lasttopic	& $(1,0,1,0,0)$	& $(2,1,1,0,0)$	& 891	& $(6,15)$ \\ 
\lasttopic	& $(2,0,0,1,0)$	& $(3,1,1,1,0)$	& 6864	& $(24,40)$ \\*
\lasttopic	& $(2,1,0,0,0)$	& $(3,1,0,0,0)$	& 1430	& $(10,20)$ \\ \hline

\topic{$D_4$}	& $(0,0,2,2)$	& $(2,2,2,0)$	& 840	& $(12,4)$ \\*
\lasttopic	& \faded $(2,0,0,2)$	& \faded $(3,1,1,1)$	& \multicolumn{2}{r}{\faded same as above \quad \qquad ~} \\*
\lasttopic	& \faded $(2,0,2,0)$	& \faded $(3,1,1,-1)$	& \multicolumn{2}{r}{\faded same as above \quad \qquad ~} \\ 
\lasttopic	& $(0,1,0,2)$	& $(2,2,1,1)$	& 567	& $(3,12)$ \\*
\lasttopic	& \faded $(0,1,2,0)$	& \faded $(2,2,1,-1)$	& \multicolumn{2}{r}{\faded same as above \quad \qquad ~} \\* 
\lasttopic	& \faded $(2,1,0,0)$	& \faded $(3,1,0,0)$	& \multicolumn{2}{r}{\faded same as above \quad \qquad ~} \\
\lasttopic	& $(0,3,0,0)$	& $(3,3,0,0)$	& 1925	& $(1,28)$ \\*
\lasttopic	& $(1,0,1,1)$	& $(2,1,1,0)$	& 350	& $(6,8)$ \\ \hline

\topic{$D_5$}	& $(0,0,0,0,4)$	& $(2,2,2,2,2)$	& 2772	& $(12,4)$ \\*
\lasttopic	& \faded $(0,0,0,4,0)$	& \faded $(2,2,2,2,-2)$	& \multicolumn{2}{r}{\faded same as above \quad \qquad ~} \\ 
\lasttopic	& $(0,0,0,1,1)$	& $(1,1,1,1,0)$	& 210	& $(6,4)$ \\ 
\lasttopic	& $(0,0,1,0,2)$	& $(2,2,2,1,1)$	& 6930	& $(24,36)$ \\*
\lasttopic	& \faded $(0,0,1,2,0)$	& \faded $(2,2,2,1,-1)$	& \multicolumn{2}{r}{\faded same as above \quad \qquad ~} \\ 
\lasttopic	& $(0,0,2,0,0)$	& $(2,2,2,0,0)$	& 4125	& $(33,12)$ \\ 
\lasttopic	& $(0,1,0,0,0)$	& $(1,1,0,0,0)$	& 45	& $(1,4)$ \\ 
\lasttopic	& $(1,0,0,0,2)$	& $(2,1,1,1,1)$	& 1050	& $(12,8)$ \\*
\lasttopic	& \faded $(1,0,0,2,0)$	& \faded $(2,1,1,1,-1)$	& \multicolumn{2}{r}{\faded same as above \quad \qquad ~} \\*
\lasttopic	& $(1,0,1,0,0)$	& $(2,1,1,0,0)$	& 945	& $(9,16)$ \\ \hline

\topic{$D_6$}	& $(0,0,0,0,0,4)$	& $(2,2,2,2,2,2)$	& 28314	& $(60,10)$ \\*
\lasttopic	& \faded $(0,0,0,0,4,0)$	& \faded $(2,2,2,2,2,-2)$	& \multicolumn{2}{r}{\faded same as above \quad \qquad ~} \\ 
\lasttopic	& $(0,0,0,0,2,2)$	& $(2,2,2,2,2,0)$	& 99099	& $(171,96)$ \\ 
\lasttopic	& $(0,0,0,1,0,2)$	& $(2,2,2,2,1,1)$	& 84942	& $(90,180)$ \\*
\lasttopic	& \faded $(0,0,0,1,2,0)$	& \faded $(2,2,2,2,1,-1)$	& \multicolumn{2}{r}{\faded same as above \quad \qquad ~} \\ 
\lasttopic	& $(0,0,0,2,0,0)$	& $(2,2,2,2,0,0)$	& 55055	& $(175,40)$ \\ 
\lasttopic	& $(0,0,1,0,1,1)$	& $(2,2,2,1,1,0)$	& 90090	& $(180,210)$ \\ 
\lasttopic	& $(0,0,2,0,0,0)$	& $(2,2,2,0,0,0)$	& 14014	& $(74,20)$ \\ 
\lasttopic	& $(0,1,0,0,0,2)$	& $(2,2,1,1,1,1)$	& 21450	& $(90,40)$ \\*
\lasttopic	& \faded $(0,1,0,0,2,0)$	& \faded $(2,2,1,1,1,-1)$	& \multicolumn{2}{r}{\faded same as above \quad \qquad ~} \\ 
\lasttopic	& $(0,1,0,1,0,0)$	& $(2,2,1,1,0,0)$	& 21021	& $(45,120)$ \\ 
\lasttopic	& $(0,3,0,0,0,0)$	& $(3,3,0,0,0,0)$	& 23100	& $(10,110)$ \\ 
\lasttopic	& $(1,0,0,0,1,1)$	& $(2,1,1,1,1,0)$	& 8085	& $(45,40)$ \\ 
\lasttopic	& $(1,0,1,0,0,0)$	& $(2,1,1,0,0,0)$	& 2079	& $(15,24)$ \\ 
\lasttopic	& $(2,0,0,0,0,2)$	& $(3,1,1,1,1,1)$	& 27027	& $(45,90)$ \\*
\lasttopic	& \faded $(2,0,0,0,2,0)$	& \faded $(3,1,1,1,1,-1)$	& \multicolumn{2}{r}{\faded same as above \quad \qquad ~} \\ 
\lasttopic	& $(2,0,0,1,0,0)$	& $(3,1,1,1,0,0)$	& 27456	& $(120,56)$ \\*
\lasttopic	& $(2,1,0,0,0,0)$	& $(3,1,0,0,0,0)$	& 2860	& $(10,30)$ \\ \hline

\topic{$E_6$}	& $(0,0,0,0,0,3)$	& $(0,0,0,0,3,-1,-1,1)$	& 3003	& $(16,8)$ \\*
\lasttopic	& \faded $(3,0,0,0,0,0)$	& \faded $(0,0,0,0,0,-2,-2,2)$	& \multicolumn{2}{r}{\faded same as above \quad \qquad ~} \\ 
\lasttopic	& $(0,0,0,0,1,1)$	& $(0,0,0,1,2,-1,-1,1)$	& 5824	& $(32,32)$ \\*
\lasttopic	& \faded $(1,0,1,0,0,0)$	& \faded $\left(-\frac{1}{2},\frac{1}{2},\frac{1}{2},\frac{1}{2},\frac{1}{2},-\frac{3}{2},-\frac{3}{2},\frac{3}{2}\right)$	& \multicolumn{2}{r}{\faded same as above \quad \qquad ~} \\ 
\lasttopic	& $(0,0,0,0,3,0)$	& $(0,0,0,3,3,-2,-2,2)$	& 1559376	& $(980,1072)$ \\*
\lasttopic	& \faded $(0,0,3,0,0,0)$	& \faded $\left(-\frac{3}{2},\frac{3}{2},\frac{3}{2},\frac{3}{2},\frac{3}{2},-\frac{5}{2},-\frac{5}{2},\frac{5}{2}\right)$	& \multicolumn{2}{r}{\faded same as above \quad \qquad ~} \\ 
\lasttopic	& $(0,0,0,1,0,0)$	& $(0,0,1,1,1,-1,-1,1)$	& 2925	& $(17,28)$ \\ 
\lasttopic	& $(0,0,1,0,0,2)$	& $\left(-\frac{1}{2},\frac{1}{2},\frac{1}{2},\frac{1}{2},\frac{5}{2},-\frac{3}{2},-\frac{3}{2},\frac{3}{2}\right)$	& 78975	& $(168,192)$ \\*
\lasttopic	& \faded $(2,0,0,0,1,0)$	& \faded $(0,0,0,1,1,-2,-2,2)$	& \multicolumn{2}{r}{\faded same as above \quad \qquad ~} \\ 
\lasttopic	& $(0,0,1,0,1,0)$	& $\left(-\frac{1}{2},\frac{1}{2},\frac{1}{2},\frac{3}{2},\frac{3}{2},-\frac{3}{2},-\frac{3}{2},\frac{3}{2}\right)$	& 70070	& $(198,176)$ \\ 
\lasttopic	& $(0,0,2,0,0,1)$	& $(-1,1,1,1,2,-2,-2,2)$	& 600600	& $(758,712)$ \\*
\lasttopic	& \faded $(1,0,0,0,2,0)$	& \faded $(0,0,0,2,2,-2,-2,2)$	& \multicolumn{2}{r}{\faded same as above \quad \qquad ~} \\
\lasttopic	& $(0,1,0,0,0,0)$	& $\left(\frac{1}{2},\frac{1}{2},\frac{1}{2},\frac{1}{2},\frac{1}{2},-\frac{1}{2},-\frac{1}{2},\frac{1}{2}\right)$	& 78	& $(2,4)$ \\*
\lasttopic	& $(1,0,0,0,0,1)$	& $(0,0,0,0,1,-1,-1,1)$	& 650	& $(12,8)$ \\  \hline

\topic{$E_7$}	& $(0,0,0,0,0,0,4)$	& $(0,0,0,0,0,4,-2,2)$	& 293930	& $(315,105)$ \\*
\lasttopic	& $(0,0,0,0,0,1,2)$	& $(0,0,0,0,1,3,-2,2)$	& 915705	& $(630,945)$ \\ 
\lasttopic	& $(0,0,0,0,0,2,0)$	& $(0,0,0,0,2,2,-2,2)$	& 617253	& $(861,336)$ \\ 
\lasttopic	& $(0,0,0,0,1,0,1)$	& $(0,0,0,1,1,2,-2,2)$	& 980343	& $(945,1134)$ \\ 
\lasttopic	& $(0,0,0,0,2,0,0)$	& $(0,0,0,2,2,2,-3,3)$	& 109120648	& $(27006,31318)$ \\ 
\lasttopic	& $(0,0,0,1,0,0,0)$	& $(0,0,1,1,1,1,-2,2)$	& 365750	& $(665,315)$ \\ 
\lasttopic	& $(0,0,1,0,0,0,0)$	& $\left(-\frac{1}{2},\frac{1}{2},\frac{1}{2},\frac{1}{2},\frac{1}{2},\frac{1}{2},-\frac{3}{2},\frac{3}{2}\right)$	& 8645	& $(21,56)$ \\ 
\lasttopic	& $(0,1,0,0,0,0,1)$	& $\left(\frac{1}{2},\frac{1}{2},\frac{1}{2},\frac{1}{2},\frac{1}{2},\frac{3}{2},-\frac{3}{2},\frac{3}{2}\right)$	& 40755	& $(120,105)$ \\ 
\lasttopic	& $(0,1,0,0,1,0,0)$	& $\left(\frac{1}{2},\frac{1}{2},\frac{1}{2},\frac{3}{2},\frac{3}{2},\frac{3}{2},-\frac{5}{2},\frac{5}{2}\right)$	& 11316305	& $(6363,5978)$ \\ 
\lasttopic	& $(0,2,0,0,0,0,0)$	& $(1,1,1,1,1,1,-2,2)$	& 253935	& $(189,414)$ \\ 
\lasttopic	& $(1,0,0,0,0,0,2)$	& $(0,0,0,0,0,2,-2,2)$	& 150822	& $(315,189)$ \\ 
\lasttopic	& $(1,0,0,0,0,1,0)$	& $(0,0,0,0,1,1,-2,2)$	& 152152	& $(210,378)$ \\*
\lasttopic	& $(3,0,0,0,0,0,0)$	& $(0,0,0,0,0,0,-3,3)$	& 238602	& $(105,399)$ \\ \hline

\topic{$E_8$}	& $(0,0,0,0,0,0,0,3)$	& $(0,0,0,0,0,0,3,3)$	& 1763125	& $(525,1240)$ \\*
\lasttopic	& $(0,0,0,0,0,0,1,0)$	& $(0,0,0,0,0,1,1,2)$	& 30380	& $(28,112)$ \\ 
\lasttopic	& $(0,0,0,0,0,1,0,0)$	& $(0,0,0,0,1,1,1,3)$	& 2450240	& $(1896,1064)$ \\ 
\lasttopic	& $(0,0,0,0,1,0,0,0)$	& $(0,0,0,1,1,1,1,4)$	& 146325270	& $(33782,30688)$ \\ 
\lasttopic	& $(0,0,0,1,0,0,0,0)$	& $(0,0,1,1,1,1,1,5)$	& 6899079264	& $(657608,699496)$ \\ 
\lasttopic	& $(0,0,1,0,0,0,0,0)$	& $\left(-\frac{1}{2},\frac{1}{2},\frac{1}{2},\frac{1}{2},\frac{1}{2},\frac{1}{2},\frac{1}{2},\frac{7}{2}\right)$	& 6696000	& $(2520,3480)$ \\ 
\lasttopic	& $(0,1,0,0,0,0,0,0)$	& $\left(\frac{1}{2},\frac{1}{2},\frac{1}{2},\frac{1}{2},\frac{1}{2},\frac{1}{2},\frac{1}{2},\frac{5}{2}\right)$	& 147250	& $(210,160)$ \\ 
\lasttopic	& $(1,0,0,0,0,0,0,1)$	& $(0,0,0,0,0,0,1,3)$	& 779247	& $(567,840)$ \\*
\lasttopic	& $(2,0,0,0,0,0,0,0)$	& $(0,0,0,0,0,0,0,4)$	& 4881384	& $(2808,1296)$ \\ \hline

\topic{$F_4$}	& $(0,0,0,3)$	& $(3,0,0,0)$	& 2652	& $(28,8)$ \\*
\lasttopic	& $(0,0,1,0)$	& $\left(\frac{3}{2},\frac{1}{2},\frac{1}{2},\frac{1}{2}\right)$	& 273	& $(1,8)$ \\ 
\lasttopic	& $(0,1,0,0)$	& $(2,1,1,0)$	& 1274	& $(6,20)$ \\ 
\lasttopic	& $(1,0,0,1)$	& $(2,1,0,0)$	& 1053	& $(9,12)$ \\*
\lasttopic	& $(3,0,0,0)$	& $(3,3,0,0)$	& 12376	& $(16,72)$ \\ \hline

\topic{$G_2$}	& $(0,3)$	& $(-3,-3,6)$	& 273	& $(1,8)$ \\*
\lasttopic	& $(1,1)$	& $(-1,-2,3)$	& 64	& $(2,2)$ \\*
\lasttopic	& $(3,0)$	& $(0,-3,3)$	& 77	& $(1,4)$ \\ \hline
\end{longtable}
\end{small}

\section{Implementation in LiE of the $w_0$-signature algorithm}
\label{sec:lie_code}

Here is the LiE code that was used to compute Tables \ref{pure_by_hand} and~\ref{mixed_by_hand}. Here is an example to illustrate the syntax of the functions defined in this code: the last line of Table~\ref{mixed_by_hand} can be computed by executing the command \texttt{w0G2([3,0])}.

Some explanation is needed about the big list of matrices that appears in this code. In order to define a branching rule from a Lie algebra $\mathfrak{g}$ to a Lie subalgebra $\mathfrak{s}$, the LiE software needs as input a ``restriction matrix''~$M$, which is defined by the identity
\[M \transpose{\Omega_{\mathfrak{s}}} = \transpose{\Omega_{\mathfrak{g}}},\]
where $\Omega_{\mathfrak{g}}$ (resp. $\Omega_{\mathfrak{s}}$) is the matrix whose columns are the vectors of the ``basis of fundamental weights'' of the Cartan subalgebra of~$\mathfrak{g}$ (resp. of~$\mathfrak{s}$). (The transpose is there because the LiE software happens to represent all vectors as row matrices, rather than column matrices as usual.)

What we call here a ``basis of fundamental weights'' is actually an arbitrary extension of the family of fundamental weights to a basis of the Cartan subalgebra. For the semisimple algebra $\mathfrak{g}$, the fundamental weights form a basis by themselves. The algebra~$\mathfrak{s}$, on the other hand, is in our case isomorphic to $\lie{sl}(2,\CC)^s \times \CC^t$. So it has only $s$ fundamental weights, namely $\frac{1}{2}\alpha_1, \ldots, \frac{1}{2}\alpha_s$ (where the $\alpha_i$ are defined in the proof of Proposition~\ref{sl2_copies} and listed in Table~\ref{tab:orthoroots}); we then choose $t$ more vectors that complete this family to a basis.

Note that reordering the columns of~$\Omega_\mathfrak{s}$ (which corresponds to reordering the columns of~$M$) does not change the result. We have chosen to put the true fundamental weights first, in the order given in Table~\ref{tab:orthoroots}; and the ``filler'' vectors last. (The order of the columns of~$\Omega_\mathfrak{g}$, on the other hand, is fixed by LiE to be the Bourbaki ordering.)

~

~

\begin{verbatim}
##########
# Computes the w0-sign of a representation of A1 with highest weight x.
##########
a1_sign(int x) = 
  if x%2 == 0
    then (-1)^(x/2)
    else 0
  fi

##########
# Computes the w0-sign of a representation of T1 (abelian Lie group of rank 1)
# with "highest weight" (or, rather, highest charge) x.
##########
t_sign(int x) =
  if x == 0
    then 1
    else 0
  fi

##########
# Computes the w0-sign of a representation of (A1)^l x (T1)^(n-l) with highest
# weight given by v. Here n is given by the length of the vector v, and l is passed
# as a separate argument.
##########
tot_sign(vec v; int l) =
  {loc res=1;
   for i=1 to l do
     res=res*a1_sign(v[i])
   od;
   for i=l+1 to size(v) do
     res=res*t_sign(v[i]) od;
   res}

##########
# The LiE software encodes a non-irreducible representation of (A1)^l x (T1)^(n-l)
# by a polynomial p, in which each monomial of the form exp(v) stands for a direct
# summand with highest weight v. This function computes the w0-signature of such a
# representation.
# Here n is given by the lengths of vectors occuring as exponents in p, and l is
# passed as a separate argument.
##########
count_signature(pol p; int l) =
  {loc res=[0,0];
   for i=1 to length(p) do
     if tot_sign(expon(p,i), l) == +1
       then res[1]+=coef(p,i)
       else if tot_sign(expon(p,i), l) == -1
         then res[2]+=coef(p,i)
       fi
     fi
   od;
   res}

##########
# Declaration of restriction matrices
##########
\end{verbatim}
\setlength{\columnseprule}{1pt}
\begin{multicols}{2}
\begin{verbatim}
resA1 = [[1]]
resA2 = [[1,1],
         [1,-1]]
resA3 = [[1,0,1],
         [1,1,0],
         [1,0,-1]]
resA4 = [[1,0,3,-2],
         [1,1,1,1],
         [1,1,-1,-1],
         [1,0,-3,2]]
resA5 = [[1,0,0,4,-2],
         [1,1,0,2,2],
         [1,1,1,0,0],
         [1,1,0,-2,-2],
         [1,0,0,-4,2]]
resA6 = [[1,0,0,5,-2,-2],
         [1,1,0,3,3,-4],
         [1,1,1,1,1,1],
         [1,1,1,-1,-1,-1],
         [1,1,0,-3,-3,4],
         [1,0,0,-5,2,2]]
resA7 = [[1,0,0,0,6,-2,-2],
         [1,1,0,0,4,4,-4],
         [1,1,1,0,2,2,2],
         [1,1,1,1,0,0,0],
         [1,1,1,0,-2,-2,-2],
         [1,1,0,0,-4,-4,4],
         [1,0,0,0,-6,2,2]]
resA8 = [[1,0,0,0,7,-2,-2,-2],
         [1,1,0,0,5,5,-4,-4],
         [1,1,1,0,3,3,3,-6],
         [1,1,1,1,1,1,1,1],
         [1,1,1,1,-1,-1,-1,-1],
         [1,1,1,0,-3,-3,-3,6],
         [1,1,0,0,-5,-5,4,4],
         [1,0,0,0,-7,2,2,2]]
resB2 = [[1,1],
         [1,0]]
resB3 = [[1,1,0],
         [2,0,0],
         [1,0,1]]
resB4 = [[1,1,0,0],
         [2,0,0,0],
         [2,0,1,1],
         [1,0,1,0]]
resB5 = [[1,1,0,0,0],
         [2,0,0,0,0],
         [2,0,1,1,0],
         [2,0,2,0,0],
         [1,0,1,0,1]]
resB6 = [[1,1,0,0,0,0],
         [2,0,0,0,0,0],
         [2,0,1,1,0,0],
         [2,0,2,0,0,0],
         [2,0,2,0,1,1],
         [1,0,1,0,1,0]]
resB7 = [[1,1,0,0,0,0,0],
         [2,0,0,0,0,0,0],
         [2,0,1,1,0,0,0],
         [2,0,2,0,0,0,0],
         [2,0,2,0,1,1,0],
         [2,0,2,0,2,0,0],
         [1,0,1,0,1,0,1]]
resB8 = [[1,1,0,0,0,0,0,0],
         [2,0,0,0,0,0,0,0],
         [2,0,1,1,0,0,0,0],
         [2,0,2,0,0,0,0,0],
         [2,0,2,0,1,1,0,0],
         [2,0,2,0,2,0,0,0],
         [2,0,2,0,2,0,1,1],
         [1,0,1,0,1,0,1,0]]
resC3 = [[1,0,0],
         [1,1,0],
         [1,1,1]]
resC4 = [[1,0,0,0],
         [1,1,0,0],
         [1,1,1,0],
         [1,1,1,1]]
resC5 = [[1,0,0,0,0],
         [1,1,0,0,0],
         [1,1,1,0,0],
         [1,1,1,1,0],
         [1,1,1,1,1]]
resC6 = [[1,0,0,0,0,0],
         [1,1,0,0,0,0],
         [1,1,1,0,0,0],
         [1,1,1,1,0,0],
         [1,1,1,1,1,0],
         [1,1,1,1,1,1]]
resC7 = [[1,0,0,0,0,0,0],
         [1,1,0,0,0,0,0],
         [1,1,1,0,0,0,0],
         [1,1,1,1,0,0,0],
         [1,1,1,1,1,0,0],
         [1,1,1,1,1,1,0],
         [1,1,1,1,1,1,1]]
resC8 = [[1,0,0,0,0,0,0,0],
         [1,1,0,0,0,0,0,0],
         [1,1,1,0,0,0,0,0],
         [1,1,1,1,0,0,0,0],
         [1,1,1,1,1,0,0,0],
         [1,1,1,1,1,1,0,0],
         [1,1,1,1,1,1,1,0],
         [1,1,1,1,1,1,1,1]]
resD4 = [[1,1,0,0],
         [2,0,0,0],
         [1,0,1,0],
         [1,0,0,1]]
resD5 = [[1,1,0,0,0],
         [2,0,0,0,0],
         [2,0,1,1,0],
         [1,0,1,0,1],
         [1,0,1,0,-1]]
resD6 = [[1,1,0,0,0,0],
         [2,0,0,0,0,0],
         [2,0,1,1,0,0],
         [2,0,2,0,0,0],
         [1,0,1,0,1,0],
         [1,0,1,0,0,1]]
resD7 = [[1,1,0,0,0,0,0],
         [2,0,0,0,0,0,0],
         [2,0,1,1,0,0,0],
         [2,0,2,0,0,0,0],
         [2,0,2,0,1,1,0],
         [1,0,1,0,1,0,1],
         [1,0,1,0,1,0,-1]]
resD8 = [[1,1,0,0,0,0,0,0],
         [2,0,0,0,0,0,0,0],
         [2,0,1,1,0,0,0,0],
         [2,0,2,0,0,0,0,0],
         [2,0,2,0,1,1,0,0],
         [2,0,2,0,2,0,0,0],
         [1,0,1,0,1,0,1,0],
         [1,0,1,0,1,0,0,1]]
resE6 = [[0,0,1,1,1,0],
         [0,0,2,0,0,0],
         [1,0,2,1,0,1],
         [1,1,3,1,0,0],
         [1,0,2,1,0,-1],
         [0,0,1,1,-1,0]]
resE7 = [[0,0,0,0,0,0,2],
         [1,0,1,0,1,0,2],
         [0,1,1,0,1,0,3],
         [0,0,2,0,2,0,4],
         [0,0,1,1,2,0,3],
         [0,0,0,0,2,0,2],
         [0,0,0,0,1,1,1]]
resE8 = [[0,0,0,0,0,0,2,2],
         [1,0,1,0,1,0,3,2],
         [0,1,1,0,1,0,4,3],
         [0,0,2,0,2,0,6,4],
         [0,0,1,1,2,0,5,3],
         [0,0,0,0,2,0,4,2],
         [0,0,0,0,1,1,3,1],
         [0,0,0,0,0,0,2,0]]
resF4 = [[2,0,0,0],
         [3,1,1,1],
         [2,1,1,0],
         [1,1,0,0]]
resG2 = [[1,1],
         [0,2]]
\end{verbatim}
\end{multicols}
\begin{verbatim}
##########
# Definition of functions computing the w0-signature (one function per Lie algebra)
##########
w0A1(vec v) = count_signature(branch(v,A1,resA1,A1),1)
w0A2(vec v) = count_signature(branch(v,A1T1,resA2,A2),1)
w0A3(vec v) = count_signature(branch(v,A1A1T1,resA3,A3),2)
w0A4(vec v) = count_signature(branch(v,A1A1T2,resA4,A4),2)
w0A5(vec v) = count_signature(branch(v,A1A1A1T2,resA5,A5),3)
w0A6(vec v) = count_signature(branch(v,A1A1A1T3,resA6,A6),3)
w0A7(vec v) = count_signature(branch(v,A1A1A1A1T3,resA7,A7),4)
w0A8(vec v) = count_signature(branch(v,A1A1A1A1T4,resA8,A8),4)
w0B2(vec v) = count_signature(branch(v,A1A1,resB2,B2),2)
w0B3(vec v) = count_signature(branch(v,A1A1A1,resB3,B3),3)
w0B4(vec v) = count_signature(branch(v,A1A1A1A1,resB4,B4),4)
w0B5(vec v) = count_signature(branch(v,A1A1A1A1A1,resB5,B5),5)
w0B6(vec v) = count_signature(branch(v,A1A1A1A1A1A1,resB6,B6),6)
w0B7(vec v) = count_signature(branch(v,A1A1A1A1A1A1A1,resB7,B7),7)
w0B8(vec v) = count_signature(branch(v,A1A1A1A1A1A1A1A1,resB8,B8),8)
w0C3(vec v) = count_signature(branch(v,A1A1A1,resC3,C3),3)
w0C4(vec v) = count_signature(branch(v,A1A1A1A1,resC4,C4),4)
w0C5(vec v) = count_signature(branch(v,A1A1A1A1A1,resC5,C5),5)
w0C6(vec v) = count_signature(branch(v,A1A1A1A1A1A1,resC6,C6),6)
w0C7(vec v) = count_signature(branch(v,A1A1A1A1A1A1A1,resC7,C7),7)
w0C8(vec v) = count_signature(branch(v,A1A1A1A1A1A1A1A1,resC8,C8),8)
w0D4(vec v) = count_signature(branch(v,A1A1A1A1,resD4,D4),4)
w0D5(vec v) = count_signature(branch(v,A1A1A1A1T1,resD5,D5),4)
w0D6(vec v) = count_signature(branch(v,A1A1A1A1A1A1,resD6,D6),6)
w0D7(vec v) = count_signature(branch(v,A1A1A1A1A1A1T1,resD7,D7),6)
w0D8(vec v) = count_signature(branch(v,A1A1A1A1A1A1A1A1,resD8,D8),8)
w0E6(vec v) = count_signature(branch(v,A1A1A1A1T2,resE6,E6),4)
w0E7(vec v) = count_signature(branch(v,A1A1A1A1A1A1A1,resE7,E7),7)
w0E8(vec v) = count_signature(branch(v,A1A1A1A1A1A1A1A1,resE8,E8),8)
w0F4(vec v) = count_signature(branch(v,A1A1A1A1,resF4,F4),4)
w0G2(vec v) = count_signature(branch(v,A1A1,resG2,G2),2)
\end{verbatim}

\section*{Acknowledgements}

We would like to thank Ernest Vinberg, who suggested the crucial idea of using Theorem~\ref{thm:Vinberg} to prove Lemma~\ref{lem:ideal}; as well as Jeffrey Adams and Yifan Wang for some interesting discussions. The second author of the present paper was supported by the National Science Foundation grant DMS-1709952.

\end{document}